\documentclass[a4paper,12pt]{amsart}
\usepackage[utf8]{inputenc}
\usepackage[T1]{fontenc}
\usepackage[UKenglish]{babel}
\usepackage[margin=18mm]{geometry}

\usepackage{color}

\usepackage{graphicx}

\usepackage{amsmath,amssymb,amsfonts,amsthm}
\usepackage{mathrsfs,eucal,dsfont}
\usepackage{verbatim,enumitem}

\usepackage{hyperref,url}

\newcommand{\R}{\mathds R}

\newcommand{\I}{\mathds 1}

\def\aa{\alpha}

\def\d{{\rm d}}
\def\<{\langle}
\def\>{\rangle}

 \def\ss{\sqrt}

\def\R{\mathbb R}   \def\ss{\sqrt} 
 \def\kk{\kappa} 
  \def\vv{\varepsilon} 
\def\<{\langle} \def\>{\rangle}  
  \def\nn{\nabla}  
\def\d{\text{\rm{d}}}  \def\aa{\alpha} 
  \def\si{\sigma} 
 \def\beq{\begin{equation}}  
 
\def\e{\text{\rm{e}}}  \def\OO{\Omega}  
  
 \def\P{\mathbb P}

\def\E{\mathbb E}

\def\to{\rightarrow}
\def\8{\infty}\def\3{\triangle}
\def\1{\lesssim}

\renewcommand{\bar}{\overline}
\renewcommand{\hat}{\widehat}
\renewcommand{\tilde}{\widetilde}

\newtheorem{theorem}{Theorem}[section]
\newtheorem{lemma}[theorem]{Lemma}
\newtheorem{proposition}[theorem]{Proposition}

\theoremstyle{definition}

\newtheorem{remark}[theorem]{Remark}

\numberwithin{equation}{section}
\begin{document}
\allowdisplaybreaks

\title[McKean-Vlasov SDEs with common noise] {Long time behavior  of one-dimensional McKean-Vlasov SDEs with common noise}

\author{
Jianhai Bao\qquad
Jian Wang}
\date{}
\thanks{\emph{J.\ Bao:} Center for Applied Mathematics, Tianjin University, 300072  Tianjin, P.R. China. \url{jianhaibao@tju.edu.cn}}

\thanks{\emph{J.\ Wang:}
School  of Mathematics and Statistics \& Key Laboratory of Analytical Mathematics and Applications (Ministry of Education) \& Fujian Provincial Key Laboratory
of Statistics and Artificial Intelligence, Fujian Normal University, 350007 Fuzhou, P.R. China. \url{jianwang@fjnu.edu.cn}}

\maketitle

\begin{abstract}
In this paper, by introducing  a new type  asymptotic coupling by reflection,  we explore  the long time behavior of random probability measure flows associated with a
large class of one-dimensional McKean-Vlasov SDEs with common noise.  Concerning the  McKean-Vlasov SDEs with common noise under consideration in the present work, in contrast to the existing literature,
 the drift terms are much more general rather than of the convolution form, and, in particular, can be of polynomial growth
 with respect to the spatial variables,
 and moreover idiosyncratic noises are allowed to be of multiplicative type. Most importantly, our main result indicates that both the common noise and the idiosyncratic noise  facilitate the exponential contractivity of the associated measure-valued
  processes.

\medskip


\smallskip

\end{abstract}
\section{Introduction and main result}
\subsection{Background}
Consider a mean-field game model with $N$ particles evolving in $\R^d$:
\begin{align}\label{EW0}
\d X_t^i=b(X_t^i,\hat\mu_t^N)\d t+\si  (X_t^i,\hat\mu_t^N)\d B_t^i,\quad i=1,\cdots,N,
\end{align}
where $\hat\mu_t^N:=\frac{1}{N}\sum_{j=1}^N\delta_{X_t^j}$ (the empirical measure of all particles) and $B^1:=(B_t^1)_{t\ge0}$, $\cdots,$
$B^N:=(B_t^N)_{t\ge0}$ are  independent $d$-dimensional  Brownian motions on some complete filtered probability space. In \eqref{EW0}, $(B^1,\cdots,B^N)$ is referred to as an idiosyncratic noise (independent from one individual to another). As we know,  the classical theory on
  propagation of chaos (see e.g. \cite{Sznitman}) demonstrates that
all individual particles   become asymptotically independent  when $N\to\8$. So,     the random   probability measure $\hat\mu_t^N$   converges to a deterministic distribution   and moreover the resulting state of a single particle is
 described by the McKean-Vlasov SDE:
\begin{align}\label{EW1}
\d X_t=b(X_t,\mu_t)\d t+\sigma(X_t,\mu_t)\d B_t,
\end{align}
where $\mu_t:=\mathscr L_{X_t}$ stands for the law of $X_t$ and $(B_t)_{t\ge0}$ is a $d$-dimensional Brownian motion. Initially,
the McKean-Vlasov SDE \eqref{EW1}  was introduced to explore nonlinear Fokker-Planck equations (FPEs for brevity) based on Kac's foundations of kinetic theory \cite{Kac}.   In the past few decades, as far as  McKean-Vlasov SDEs are concerned, significant advancements have been made  on
behaviours in a finite-time horizon (e.g. strong/weak well-posedness \cite{dST,HSS,HW,Wangc} and numerical approximations \cite{dES,dSTT}) and
long-time asymptotics (e.g.  ergodicity \cite{But,Eberle,LMW,Wanga,Wangb} and uniform-in-time propagation of chaos \cite{CST,DEGZ,GLM,GM,Schuh}).

 Nevertheless, in some circumstances,   the individual particles of a mean-field game model are subject to not only idiosyncratic noises  but also
 random shocks common to all particles. On this occasion, the evolution of underlying particles cannot be modelled by \eqref{EW0} any more, and in turn  is characterized by the following mean-field SDEs:
  \begin{align}\label{EW4}
\d X_t^i=b(X_t^i,\hat\mu_t^N)\d t+\si  (X_t^i,\hat\mu_t^N)\d B_t^i+\si_0  (X_t^i,\hat\mu_t^N)\d W_t\,\quad i=1,\cdots,N,
\end{align}
 where the quantities $(b,\sigma, \hat\mu_t^N)$ and $(B^1,\cdots,B^N)$ are defined  as in \eqref{EW0}, and $(W_t)_{t\ge0}$ is a $d$-dimensional Brownian motion.
In \eqref{EW4},  $(B^1,\cdots,B^N)$   is called an idiosyncratic noise (or  individual noise)   as in  \eqref{EW0}, and $(W_t)_{t\ge0}$ is named as a common noise, which accounts for the common environment where the individual particles survive.
 In the aforementioned setting, all particles are not asymptotically independent any more  and the random empirical measure related to particles no longer  converges to a deterministic distribution as the particle number goes to infinity.
 Whereas,   the phenomenon on conditional propagation of chaos  (see e.g. \cite[Theorem 2.12]{CD2})
 illustrates that
 all particles are asymptotically  independent and the corresponding empirical distribution converges to   the common conditional distribution of each particle  conditioned on the $\sigma$-algebra  associated with the common noise.  Moreover, the subsequent limiting state of each particle can be governed by the McKean-Vlasov SDE with common noise:
\begin{align}\label{EW2}
\d X_t=b(X_t,\mu_t)\d t+\sigma(X_t,\mu_t)\d B_t+\si_0(X_t,\mu_t)\d W_t,
\end{align}
where $\mu_t:=\mathscr L_{X_t|\mathscr F_t^W}$ (the  conditional distribution given the $\sigma$-algebra $\mathscr F_t^W:=\sigma\{W_s:s\le t\}$);  $(B_t)_{t\ge0}$ and $(W_t)_{t\ge0}$ are $d$-dimensional Brownian motions. So far,  McKean-Vlasov SDEs with common noise, which, in the literature, are also called conditional McKean-Vlasov SDEs (see e.g. \cite[Chapter 2]{CD2}),   have been applied considerably  in   stochastic optimal   control and mean-field games \cite{CD2,Pham},  and
inter-bank borrowing and lending systems \cite{BLY,LS}, to name a few.  In fact, the conditional McKean-Vlasov SDE \eqref{EW2}  arises from  many practical applications  as shown in e.g. \cite{NYH,Wangd}. In detail,
in order  to construct diffusion processes generated by second order differentiable operators on the Wasserstein space, F.-Y. Wang \cite{Wangd} introduced an   image dependent SDEs, which can  indeed be reformulated  as a special conditional  McKean-Vlasov SDE  (with $\si\equiv0$ in \eqref{EW2}).  Moreover,  in \cite{NYH}, the authors explored   a mean-field game problem
with $N$ players in a random  environment, which is delineated by a continuous-time Markov chain in lieu of the usual diffusions. In particular,
they confirmed  that the associated mean-field limit process solves a conditional McKean-Vlasov SDE, where the Markov chain involved acts as the common noise.

In contrast to McKean-Vlasov SDEs without common noise, the research on McKean-Vlasov SDEs with common noise is not too rich.  Yet, in the past few years, there are still some progresses on qualitative and quantitative analyses; see, for example, \cite{BLM,HSSb,STW} on well-posedness, and \cite{CD2,ELL,Huang,STW} concerned with
finite-time conditional propagation of chaos.  According to \cite[p.\ 110-112]{CD2}, the random distribution flow $(\mu_t)_{t>0}$ associated with \eqref{EW4} solves the nonlinear  FPE:
\begin{equation}\label{EW3}
\d\mu_t=\Big(-\mbox{div}\big( b(\cdot,\mu_t)\mu_t \big)
+\frac{1}{2}\mbox{trace}\big(\nn^2((\si\si^*)(\cdot,\mu_t)\mu_t)\big)-\mbox{div}\big(\big(\si_0(\cdot,\mu_t)\d W_t\big)\mu_t\big)\Big)\d t,
\end{equation}
which is understood  in the weak sense. With regard to the well-posedness of \eqref{EW3}, we   refer to e.g. \cite{CG,FG,LPS} and references within.
Recently, via establishing superposition principles,
\cite{LSZ}  built a one-to-one correspondence between the conditional McKean-Vlasov SDE \eqref{EW2} and the stochastic FPE \eqref{EW3}.
Moreover,  the stochastic PDE \eqref{EW3} is also  linked closely to  the stochastic scalar conservation laws in the Stratonovich form as demonstrated in \cite[Appendix]{CG}. Herein, we would like to mention \cite{DWZZ}, where   Freidlin–Wentzell-type large deviation principles were addressed via the weak convergence approach,  for first-order scalar conservation laws with stochastic forcing.
  Based on the point of view above,  the research on
  the   long-time   behavior  of the random distribution flow corresponding to \eqref{EW2} amounts to the investigation on long-term asymptotics of certain kinds of stochastic FPEs   and stochastic scalar conservation laws.

No matter what the conditional McKean-Vlasov SDE \eqref{EW2} or the nonlinear FPE \eqref{EW3}, most of the existing literature (mentioned above) focuses on finite-time behaviours (e.g. well-posedness and conditional propagation of chaos).
  Nevertheless, the   asymptotic analysis in an infinite-time horizon is extremely rare.
 By comparing \eqref{EW1} with \eqref{EW2}, one of
 remarkable distinctness between them lies in that the deterministic   flow $(\mu_t)_{t>0}$ in \eqref{EW1} satisfies a deterministic nonlinear FPE whereas the random counterpart in \eqref{EW2}   fulfils a stochastic nonlinear FPE. This essential discrepancy
brings about major  challenges to tackle  the long-time behavior of the measure-valued process $(\mu_t)_{t>0}$ solving \eqref{EW3}.

  With regarding to \eqref{EW2} with $\si=0$, \cite{Wangd} treated the exponential ergodicity of the Markov process $(X_t,\mu_t)_{t\ge0}$ provided that  the drift $b$
is globally dissipative with respect to the spatial variables.
Furthermore,  for a special form of
\eqref{EW2} (or \eqref{EW3}), \cite{Maillet} tackled  the long-term distribution asymptotics for  the  conditional McKean-Vlasov SDE on $\R$:
\begin{align}\label{EW5}
\d X_t=-\bigg(V'(X_t) +\int_\R W'(X_t-y)\mu_t(\d y)\bigg)\d t+\si\d B_t+\si_0\d W_t,
\end{align}
where $\mu_t:=\mathscr L_{X_t|\mathscr F_t^{W}}$, $\si,\si_0\in\R$, and $(B_t)_{t\ge0}$ and $(W_t)_{t\ge0}$ are $1$-dimensional Brownian motions.
In \cite{Maillet}, by  designing  an approximate reflection coupling with respect to the common noise part,
the exponential ergodicity of the measure-valued process $(\mu_t)_{t>0}$ under $L^1$-Wasserstein distance was investigated. We emphasize that
$V'$ and $W'$ in \eqref{EW5} are imposed to be globally Lipschitz, and moreover that   the initial distribution of $X_0$ is required  to have a finite {\it fourth order} moment. Generally speaking,  it is quite natural to assume that the initial distribution has a finite first order moment. Whence, the requirement on a finite {\it fourth order} moment concerning the initial distribution  is very strict when the ergodicity of $(\mu_t)_{t>0}$ is discussed under the $L^1$-Wasserstein distance. Moreover,  as shown in \cite{Maillet},
the common noise part is beneficial to  ergodicity  and restoration of uniqueness on invariant probability measures for the measure-valued process $(\mu_t)_{t>0}$ whenever the intensity of the idiosyncratic noise is small enough. Regarding the idiosyncratic noise term, the synchronous coupling was applied therein so no contributions were made on the ergodic behavior of $(\mu_t)_{t>0}$ even though the intensity of the idiosyncratic noise was big enough.

\subsection{Main result}
Inspired by the aforementioned literature, in the present work, we make an attempt to  investigate the ergodic property of the measure-valued process $(\mu_t)_{t>0}$ associated with the following conditional McKean-Vlasov SDE on $\R$:
\begin{equation}\label{E1}
\d X_t= b(X_t,\mu_t) \,\d t+  \si(X_t) \,\d   B_t+\si_0\,\d W_t.
\end{equation}
Herein,  $$b :\R \otimes\mathscr P(\R)\to\R, \quad \si :\R\to\R ,   \quad \si_0\in\R, $$ where  $\mathscr P(\R)$ is   the family of probability measures on $\R$;
$(B_t)_{t\ge0}$ and $(W_t)_{t\ge0}$ are $1$-dimensional Brownian motions, where the  corresponding  probability spaces will be specified  explicitly later;
$\mu_t:=\mathscr L_{X_t|\mathscr F_t^{W}}$ is the regular
conditional distribution of $X_t$ given the $\sigma$-algebra $\mathscr F_t^{W}$. We assume that the initial value $X_0$ is an $\mathscr F_0^1$-measurable random variable, so $(W_t)_{t\ge0}$ is the solely  common source of noise.

Regarding the goal  on  exponential ergodicity under the $L^1$-Wasserstein distance for the measure-valued Markov process $(\mu_t)_{t>0}$ corresponding to \eqref{E1},
 we shall
\begin{itemize}
\item allow   the drift $b$ to be much more general rather than of the convolution form, and to be of polynomial growth, and permit the idiosyncratic noise to be of multiplicative type;

\item require the initial distribution to admit  a finite first order moment instead of higher order moments;

\item establish a novel asymptotic coupling by reflection, which is not only applied to the common noise part but also to the   idiosyncratic noise, so that  the idiosyncratic noise  can also make contributions to the ergodic behavior of $(\mu_t)_{t>0}$.
\end{itemize}
The preceding highlights are the important  source  impelling us to carry out the present work and can also be regarded as the
 main contributions of the whole paper.

To proceed, we write   down the underlying probability space we are going to work  on
and introduce some notations.
Let $(\OO^1, \mathscr F^1, \P^1)$ and $(\OO^0, \mathscr F^0,\P^0)$ be complete probability spaces with the respective filtrations $(\mathscr F_t^1)_{t\ge0}$ and $(\mathscr F_t^0)_{t\ge0}$, which satisfy the usual  hypotheses. Write
 $(B_t)_{t\ge0}$ and   $(W_t)_{t\ge0}$    as    $1$-dimensional Brownian motions supported   on   $(\OO^1,\mathscr F^1,  \P^1)$  and
$(\OO^0,\mathscr F^0, \P^0)$, respectively.  Throughout this paper, we focus on the product  probability space $(\OO, \mathscr F, \mathbb F,\P)$,
where $\OO:=\OO^0\times \OO^1$, $(\mathscr F, \P)$ is the completion of $(\mathscr F^0\otimes \mathscr F^1, \P^0\otimes \P^1)$ and $\mathbb F$ is the complete and right-continuous augmentation of $(\mathscr F^0_t\otimes \mathscr F^1_t)_{t\ge0}$.  For any $p>0$,  let $(\mathscr P_p(\R), \mathbb W_p)$ be the Wasserstein space, where
$$\mathscr P_p(\R^d):=\bigg\{\mu\in\mathscr P(\R^d): \mu(|\cdot|^p):=\int_{\R^d}|x|^p\mu(\d x)<\8\bigg\},$$
and the $L^p$-Wasserstein distance is defined by
\begin{align*}
\mathbb W_p(\mu,\nu) =\inf_{\pi\in\mathscr C(\mu,\nu)}\bigg(\int_{\R^d\times\R^d}|x-y|^p\pi(\d x,\d y)\bigg)^{\frac{1}{1\vee p}},\quad \mu,\nu\in\mathscr P_p(\R^d).
\end{align*}
Herein,  $\mathscr C(\mu,\nu)$ is the set of couplings between $\mu$ and $\nu.$

To investigate the ergodic property of the measure-valued process $(\mu_t)_{t>0}$, we
assume that
\begin{enumerate}\it
\item[{\rm(${\bf H}_{b,1}$)}]
$b(\cdot,\delta_{0})$ is continuous and locally bounded on $\R$, and
there exist constants $\lambda_1,\lambda_2,\lambda_3>0$ and $\ell_0\ge1$ such that for all $x,y\in\R $ and $\mu,\nu\in\mathscr P_1(\R)$,
\begin{equation}\label{E4}
\begin{split}
2(x-y) (b(x,\mu)-b(y,\mu)) &\le (\lambda_1+\lambda_2)| x-y |^2\I_{\{| x-y |\le\ell_0\}} -\lambda_2| x-y |^2
\end{split}
\end{equation}
and
\begin{align}\label{EE1}
|b(x,\mu)-b(x,\nu)|\le \lambda_3\mathbb W_1(\mu,\nu).
\end{align}

\item[{\rm(${\bf H}_{b,2}$)}]
 for any conditionally independent and identically distributed   $(X^i_t)_{1\le i\le N}$ under the filtration $\mathscr F_t^W$,
 there exists a function $\varphi:[0,\8)\to[0,\8)$ with $\lim_{r\to\8}\varphi(r)=0$ such that
\begin{align}\label{E23}
\max_{1\le i\le N}\sup_{t\ge0}\E|b(X^i_t,\mu_t^i)-b(X^i_t,\tilde \mu^{N,i}_t)| \le \varphi(N),
\end{align}
 where $\mu_t^i:=\mathscr L_{X_t^i|\mathscr F_t^W}$ and $\tilde \mu^{N,i}_t:=\frac{1}{N-1}\sum_{j=1:j\neq i}^N\delta_{X^j_t}$.

\item[{\rm(${\bf H}_\sigma$)}]  there exist   constants   $L_\sigma,\kk_{\sigma,1},\kk_{\sigma,2}>0$    such that for all $x,y\in\R,$
\begin{equation*}
 |\si(x)-\si(y) | \le L_\sigma  |x-y|,\qquad \kk_{\sigma,1}\le \si(x)^2\le \kk_{\sigma,2}.
\end{equation*}
\end{enumerate}

 In recent years, the strong well-posedness of conditional McKean-Vlasov SDEs has been treated in various scenarios provided that  the drift and diffusion terms are continuous under the $L^2$-Wasserstein distance with respect to the measure arguments; see e.g. \cite[Proposition 2.8]{CD2}  under globally Lipschitz continuity   and \cite[Theorem 2.1]{CNRS} when the underlying coefficients are of superlinear growth.  Under (${\bf H}_{b,1}$) and (${\bf H}_\sigma$), via  the fixed point iteration method  adopted in \cite[Theorem 2.1]{CNRS}, the SDE \eqref{E1} is strongly well-posed (even for the multidimensional setting, i.e., $d\ge2$),
  where the drift term involved is uniformly continuous under the $L^1$-Wasserstein distance.

  Below, we make some comments concerned with  the Assumptions (${\bf H}_{b,1}$), (${\bf H}_{b,2}$) and (${\bf H}_\sigma$).
  \begin{remark}
The condition \eqref{E4} shows that the drift $b$ is dissipative in the long distance with respect to the spatial variables. \eqref{EE1} demonstrates that $b$ is uniformly continuous under the $L^1$-Wasserstein distance. This, in addition to  (${\bf H}_{b,2}$), will be used in handling the asymptotic  propagation of chaos
in an infinite-time horizon
(see Proposition \ref{pro3} below for more details). In the Appendix section, we further provide a sufficiency to guarantee (${\bf H}_{b,2}$).  Moreover, the non-degenerate property of $\si$ plays a crucial role in constructing the asymptotic coupling by reflection, as stated in the second paragraph of Section \ref{sec3}.
\end{remark}

Before we present the main result, we introduce some additional notation.
Let for $p\ge1,$
\begin{align*}
L_p(\mathscr P(\R^d)):=\bigg\{\mu\in\mathscr P(\mathscr P(\R^d)):\int_{\mathscr P(\R^d)}\mathbb W_p(\nu,\delta_0)^p\mu(\d\nu)=\int_{\mathscr P(\R^d)}\int_{\R^d}|x|^p\nu(\d x)\mu(\d\nu)<\8\bigg\}
\end{align*}
and
\begin{align*}
\mathcal W_p(\mu,\nu):=\inf_{\pi\in\mathscr C(\mu,\nu)}\int_{\mathscr P(\R^d)\times\mathscr P(\R^d)}\mathbb W_p(\tilde \mu,\tilde\nu)\pi(\d\tilde\mu,\d\tilde\nu),\quad \mu,\nu\in L_p(\mathscr P(\R^d)).
\end{align*}

Let $(X_t)_{t\ge0}$ be the unique strong solution to the SDE \eqref{E1} so that the distribution of the starting point $X_0$ is $\mu\in \mathscr P_1(\R)$. For every $t>0$, let $\mu_t =\mathscr L_{X_t|\mathscr F_t^{W}}$ be the regular
conditional distribution of $X_t$ given the $\sigma$-algebra $\mathscr F_t^{W}$.
The main result in this paper is stated as follows.
\begin{theorem}\label{thm1}
Assume $({\bf H}_{b,1})$,    $({\bf H}_{b,2})$,  and $({\bf H}_\sigma)$.
Then, there exists a constant $\lambda_3^*>0$ such that for
all $\lambda_3\in[0,\lambda_3^*]$, there are constants $C,\lambda>0$ so that for all  $t>0$  and $\mu,\nu\in
\mathscr{P}_1(\R)$,
\begin{equation}\label{EE8}
\mathcal W_1(\mu_t,\nu_t)\le C\e^{-\lambda t}
\mathbb W_1(\mu,\nu).
\end{equation}
\end{theorem}

Now, we make an explanation  on the alternative of the initial value $X_0$.
\begin{remark}
  In the present paper, to highlight that the noise part $(W_t)_{t\ge0}$ is the unique common noise, we assume that the initial value $X_0$ is supported on the probability space $(\OO^1, \mathscr F^1_0, \P^1)$. This results in that the right hand side of \eqref{EE8} is
  $\mathbb W_1(\mu,\nu)$ rather than $\mathcal W_1(\mu,\nu)$. When (i) $X_0$ is defined on $(\OO^0, \mathscr F^0_0, \P^0)$
and (ii)  $X_0$ is measurable with respect to $\si(X_0^0,X_0^1)$ with $X_0^0$ and $X_0^1$ being defined on $(\OO^0, \mathscr F^0_0, \P^0)$ and $(\OO^1, \mathscr F^0_1, \P^1)$,
 $\mu_t $ is a version of the conditional law of $X_t$ given $(X_0,W)$ and  $(X_0^0,W)$, respectively; see \cite[Remark 2.10]{CD2} for more discussions on various choices on the initial value for McKean-Vlasov SDEs with common noise. For the case (i), $(X_0,W)$ is called the ``initial condition-common noise''; for the case (ii), $(X_0^0,W)$ plays the role of systemic noise. Concerning both cases (i) and (ii), the term $\mathbb W_1(\mu,\nu)$ on the right hand side of \eqref{EE8} can be replaced by $\mathcal W_1(\mu,\nu)$ so that \eqref{EE8} can be written in a symmetric form, i.e.,  $\mathcal W_1(\mu_t,\nu_t)\le C\e^{-\lambda t}
\mathcal W_1(\mu,\nu)$.
 \end{remark}

To establish the exponential contractivity for McKean-Vlasov SDEs without common noise, which are strongly well-posed,
one usually makes use of their decoupled versions (which are derived by freezing the measure variables). However, this idea does not work for the McKean-Vlasov SDEs with common noise due to the essentially  different roles played by the common noise and the idiosyncratic noise; see, for instance,  \cite[Theorem 2.1]{CNRS} concerning the proof of well-posedness for the McKean-Vlasov SDEs with common noise. Instead, we turn to the non-interacting particle system and the corresponding interacting particle system associated with \eqref{E1}.

The comparison between Theorem \ref{thm1} and the main result for the case $d=1$ in \cite[Section 4]{Maillet} are to be presented in Remark \ref{R:3.6}.
 So far, one might be a little bit confused why we are confined to the $1$-dimensional SDE \eqref{E1} rather than the multi-dimensional version. Below, we go into detail about the corresponding reasons.

\begin{remark} Theorem \ref{thm1} is only concerned with one-dimensional McKean-Vlasov SDEs with common noise. Though one key ingredient for the  proof of Theorem \ref{thm1} is the asymptotic coupling by reflection that can be constructed for all $d\ge1$ (see Proposition \ref{pro1} below), there are several essential difficulties to realize this idea in the high dimensional setting (i.e., $d\ge2$). In particular, for the case  $d\ge2$, the asymptotic coupling by reflection constructed in Subsection \ref{section2.2}  for the interacting particle system associated with the original McKean-Vlasov SDEs with common noise is
determined by the average difference between the component processes (see Remark \ref{remark} below for more details).
With such an asymptotic coupling by reflection,  if the drift term $b$ enjoys a very special structure,
 one can  derive  merely the estimate on the quantity $\E \mathbb W_1(\bar{{\bf X}_t^N},\bar{{\bf X}_t^{N,N}})$, where $\bar{{\bf X}_t^N}$ (resp. $\bar{{\bf X}_t^{N,N}}$) indicates the arithmetic mean of the non-interacting particles
(resp. interacting particles). Furthermore,  by following the line in \cite[Section 5]{Maillet}, to achieve the main result in Theorem \ref{thm1} for the high dimensional setting,
one needs  to quantify the difference between each component of the interacting particle system and its averaged process. To this end, a very strict condition $\si\equiv0$ need to be imposed.
  When the idiosyncratic noise vanishes, that is, the McKean-Vlasov SDEs is only driven by common noise, the corresponding problem has been treated in \cite[Section 5]{Maillet}.
On the other hand, when the coefficients corresponding to  the  McKean-Vlasov SDEs with common noise are dissipative, one can directly apply the synchronous coupling and  bypass the obstacles mentioned above; see also  \cite[Section 3]{Maillet} for the details.
 More interpretations related to  the restriction on the dimension $d=1$ will be elaborated in Remark \ref{remark:coup} and Remark \ref{remark}.
 \end{remark}

\smallskip

The rest of this paper is arranged as follows. In the next section, we establish conditional propagation of chaos in a finite-time horizon for the McKean-Vlasov SDEs with common noise, and construct an asymptotic coupling by reflection for the associated non-interacting particle system and interacting particle system. The results in Section \ref{sec2} hold for general settings (in particular for all dimensions). The asymptotic coupling by reflection
here is new in the sense that we only consider the approximation of the coupling by reflection for the interacting particle system and keep the  non-interacting particle untouched, which ensures the essential characterization of the McKean-Vlasov SDEs with common noise. The price to pay is that one needs some efforts to verify  tightness of
the asymptotic coupling process constructed via
the asymptotic coupling  by reflection.
 Section \ref{sec3} is devoted to the proof of Theorem \ref{thm1}, which is based on an explicit convergence rate of  the asymptotic   conditional propagation of chaos   in the long-time horizon   for the conditional McKean-Vlasov SDE. Most importantly, the  explicit convergence rates indicate satisfactorily that both common noise and idiosyncratic noise   facilitate the exponential contractivity of the associated measure-valued  processes.

\section{Preliminaries}\label{sec2}
 Let $(B_t^1)_{t\ge0}$ and  $(B_t^2)_{t\ge0}$    be    $d$-dimensional Brownian motions defined   on  the complete probability space   $(\OO^1,\mathscr F^1,  \P^1)$, and  $(W_t)_{t\ge0}$ be a     $d$-dimensional Brownian motion     on  the complete probability space
$(\OO^0,\mathscr F^0, \P^0)$.
   $\E$, $\E^0$ and $\E^1$ stand for  the expectation operators under the probability measures $\P:=\P^0\times \P^1,$ $\P^0$ and $\P^1$, respectively.
In this section, we focus on the   McKean-Vlasov SDE with common noise in the following form:
\begin{equation}\label{W1}
\d X_t=b(X_t,\mu_t)\d t+  \si_1\d B^1_t+\bar\si(X_t)\d B_t^2+\si_0\d W_t,
\end{equation}
where
$$b:\R^d\times\mathscr P(\R^d)\to\R^d,\quad \bar\si:\R^d\to\R^d\otimes\R^d,\quad \si_0,\sigma_1\in\R;$$
 $\mu_t:=\mathscr L_{X_t|\mathscr F_t^{W}}$;   the initial value $X_0$ is an $\mathscr F_0^1$-measurable random variable. As the chapter unfolds,
 the reason why we prefer the SDE formulated in the framework  \eqref{W1} will become more and more transparent; see, in particular, the introductory part  of the
 Section \ref{sec3}.

\ \

We shall suppose that
\begin{enumerate} \it
\item[{\rm(${\bf A}_b$)}] $b(\cdot,\delta_0):\R^d  \to\R^d $  is continuous and locally bounded on $\R^d $; there exist     constants  $L_1,L_2 >0$ such that for all $x,y\in\R^d $ and $\mu,\nu\in\mathscr P_1(\R^d)$,
\begin{equation}\label{W2}
2\<x-y, b(x,\mu)-b(y,\mu)\>\le L_1|x-y|^2
\end{equation}
and
\begin{align}\label{W3}
|b(x,\mu)-b(x,\nu)|\le L_2\mathbb W_1(\mu,\nu).
\end{align}

\item[{\rm(${\bf A}_{\bar\sigma}$)}] there exists a constant $L_3>0$ such that
\begin{align}\label{W4}
\|\bar \si(x)-\bar\si  (y )\|_{\rm HS}\le L_3|x-y|,\quad x,y\in\R^d.
\end{align}

\end{enumerate}

Under the Assumptions (${\bf A}_b$) and (${\bf A}_{\bar\sigma}$), for all $x,y\in\R^d $ and $\mu,\nu\in\mathscr P_1(\R^d)$,
\begin{equation}\label{W2--}
2\<x-y, b(x,\mu)-b(y,\nu)\>\le L_4\big(|x-y|+\mathbb W_1(\mu,\nu)\big)|x-y|,\end{equation}
where $L_4:=\max\{L_1,2L_2\}$.  Then,
the SDE \eqref{W1} under consideration has a unique
strong solution; see, for instance, the proof of \cite[Theorem 2.1]{CNRS} for related details.
To handle the theory on propagation of chaos concerned  with \eqref{W1}, we need to explore
the non-interacting particle system and the interacting particle system
 associated with \eqref{W1}, which  are  described respectively as follows: for any $i\in \mathbb S_N:=\{1,2,\cdots,N\}$,
 \begin{equation}\label{E6}
\d X_t^i= b(X_t^i,\mu_t^i) \d t+   \si_1\d B_t^{1,i}+\bar\si(X_t^i) \d   B_t^{2,i}+\si_0\d W_t,
\end{equation}
and
\begin{equation}\label{E7}
\d X_t^{i,N}=b(X_t^{i,N},\hat\mu_t^N) \d t+  \si_1\d B_t^{1,i}+\bar\si(X_t^{i,N}) \d   B_t^{2,i}+\si_0\d W_t,
\end{equation}
where, for each $i\in\mathbb S_N$,
 $ \mu_t^i:=\mathscr L_{X_t^i|\mathscr F_t^{W}} $ and
 $\hat\mu_t^{N}:=\frac{1}{N}\sum_{j=1}^N\delta_{X_t^{j,N}}$, the empirical distribution of the individual states at time $t$; the idiosyncratic noises
 $(B^{1,i})_{i\in \mathbb S_N}$ and $( B^{2,i})_{i\in \mathbb S_N}$,  with $B^{k,i}:=(B_t^{k,i})_{t\ge0}$ for all $k=1,2$ and $i\in  
  \mathbb S_N$, are mutually  independent $d$-dimensional  Brownian motions supported  on the filtration probability space $(\OO^1, \mathscr F^1, (\mathscr F^1_t)_{t\ge0},\P^1)$, and the common noise $(W_t)_{t\ge0}$, carried on the filtration probability space $(\OO^0, \mathscr F^0, (\mathscr F^0_t)_{t\ge0},\P^0)$,  is kept untouched as in \eqref{W1};
 $(X_0^{i}, X_0^{i,N})_{1\le i\le N}$
  are i.i.d.\ $\mathscr F_0^1$-measurable random variables.
 Note that \eqref{E7} can be reformulated as a classical $(\R^{d})^N$-valued SDE.  Since,  under   (${\bf A}_b$) and (${\bf A}_{\bar \sigma}$), the underlying  SDE \eqref{E7} satisfies the so-called locally weak monotonicity and globally weak coercivity, \eqref{E7} is strongly well-posed; see, for instance, \cite[Theorem 3.1.1]{PR}.

\subsection{Conditional propagation of chaos in a finite-time horizon}
Concerning the SDE \eqref{W1}, in this subsection,
we handle the  phenomenon on conditional propagation of chaos in a finite-time horizon. In the past few years, this subject  has achieved  some progresses;
see, for example, \cite[Theorem 2.12]{CD2} and \cite[Theorem 2.3]{Huang}, where the drift term and the diffusion term are Lipschitz continuous  with respect to the spatial variables, and \cite[Proposition 2.1]{CNRS}, in which the coefficients satisfy the monotone condition with respect to the spatial variables. It is worthy to emphasize that the coefficients  of  McKean-Vlasov SDEs with common noise under investigation  in \cite{CD2,CNRS,Huang} are $L^2$-Wasserstein Lipschitz continuous with respect to the measure variables. Yet, in the present  paper, the drift parts of the conditional   McKean-Vlasov SDEs we are interested in are $L^1$-Wasserstein Lipschitz continuous. In particular,
the corresponding convergence rate of conditional propagation of chaos was revealed in \cite{CD2,CNRS,Huang} whenever the initial distributions enjoy high order moments.  As far as we are concerned, the quantitative convergence rate of conditional propagation of chaos is unnecessary for our purpose, so the high order moment  of the  initial distribution is dispensable as showed in the following proposition.

\begin{proposition}\label{pro0} Consider the SDEs \eqref{E6} and \eqref{E7} with $X_0^{i,N}=X_0^i$ for all $1\le i\le N$.
Assume $({\bf A}_b)$ and $({\bf A}_{\bar\si})$, and suppose further that $\E|X_0^1|<\8$.
Then, for each given $t\ge0$ and  $i\in\mathbb S_N,$
\begin{align}\label{E9}
\lim_{N\to\8}\E\mathbb W_1(\mu_t^i,\tilde\mu_t^N)=0,
\end{align}
where $\tilde\mu_s^N:=\frac{1}{N}\sum_{j=1}^N\delta_{X_s^j}$,
and
\begin{align}\label{E13}
\lim_{N\to\8}\E|X_t^i-X_t^{i,N}|=0.
\end{align}
\end{proposition}

\begin{proof} The proof is split into two parts.

(i) First of all,  we show that for each given $t\ge0$ and $i\in\mathbb S_N,$
\begin{align}\label{EE3}
\E  |Z_t^{i,N}|
&\le   \frac{1}{2}L_4 t\e^{(L_4+L_3^2/2)t} \int_0^t\E\mathbb W_1(\mu_s^i,\tilde\mu_s^N)\,\d s ,
\end{align}
where $Z_t^{i,N}:=X_t^i-X_t^{i,N}$ and $L_4$ was given in \eqref{W2--}. Once \eqref{EE3} is verifiable, by Fatou's lemma, we deduce that
\begin{align*}
\limsup_{N\to\8}\E  |Z_t^{i,N}|&\le    \frac{1}{2}L_4 t\e^{(L_4+L_3^2/2)t}\limsup_{N\to\8}\int_0^t\E\mathbb W_1(\mu_s^i,\tilde\mu_s^N)\,\d s \\
&\le    \frac{1}{2}L_4 t\e^{(L_4+L_3^2/2)t}\int_0^t\limsup_{N\to\8}\E\mathbb W_1(\mu_s^i,\tilde\mu_s^N)\,\d s.
\end{align*}
Consequently, \eqref{E13} follows by taking \eqref{E9} into consideration.

In the sequel, we shall fix the index  $i\in\mathbb S_N.$
For any $\delta\in (0,1]$, define the function $V_\delta$ by
\begin{align}\label{W0}
V_\delta(x)=(\delta+|x|^2)^{{1}/{2}},\quad x\in\R^d,
\end{align}
which indeed  is  a smooth approximation of the function $\R^d\ni x\mapsto |x|$.
Applying It\^o's formula and utilizing the facts:
\begin{align}\label{W**}
\nn V_\delta(x)=\frac{x}{V_\delta(x)}, \quad   \quad \nn^2V_\delta(x)= \frac{1}{V_\delta(x)}I_d-  \frac{ x \otimes x }{V_\delta(x)^3},\quad x\in\R^d,
\end{align}
 we deduce from \eqref{W4} and \eqref{W2--} that
\begin{equation*}
\begin{split}
\d V_\delta(Z_t^{i,N})&=\big\<\nn V_\delta(Z_t^{i,N}), (b(X_t^i,\mu_t^i)-b(X_t^{i,N},\hat\mu_t^N)\big\>\d t\\
&\quad+\frac{1}{2}\big\<\nn^2 V_\delta(Z_t^{i,N}),  (\bar\si(X_t^i)-
\bar\si(X_t^{i,N}) )(\bar\si(X_t^i)-
\bar\si(X_t^{i,N}) )^*\big\>_{\rm HS} \d t+\d M_t^{i,N}\\
&\le \frac{|Z_t^{i,N}|}{2 V_\delta(Z_t^{i,N})}\big((L_4+L_3^2) |Z_t^{i,N}|+L_4\mathbb W_1(\mu_t^i,\hat\mu_t^N)\big) \d t+ \d M_t^{i,N}\\
&\le \frac{1}{2}\big((L_4+L_3^2)|Z_t^{i,N}|+L_4\mathbb W_1(\mu_t^i,\hat\mu_t^N)\big)\d t+ \d M_t^{i,N},
\end{split}
\end{equation*}
where
$$  \d M_t^{i,N}:= \big\<\nn V_\delta (Z_t^{i,N}),  \big(\bar\si(X_t^i)-
\bar\si(X_t^{i,N}) \big)\d B_t^{2,i}\big\>.$$
Thus, via Fatou's lemma, in addition to $X_0^i=X_0^{i,N}$, we have
\begin{align*}
\E  |Z_t^{i,N}|\le \frac{1}{2}\int_0^t\big((L_4+L_3^2)\E|Z_s^{i,N}|+L_4\E\mathbb W_1(\mu_s^i,\hat\mu_s^N)\big)\d s.
\end{align*}
Note from the triangle inequality that
\begin{align*}
\mathbb W_1(\mu_t^i,\hat\mu_t^N)&\le \mathbb W_1(\mu_t^i,\tilde\mu_t^N)+\mathbb W_1(\tilde\mu_t^N,\hat\mu_t^N)
\le \mathbb W_1(\mu_t^i,\tilde\mu_t^N)+\frac{1}{N}\sum_{j=1}^N|Z_t^{j,N}|,
\end{align*}
since $\frac{1}{N}\sum_{j=1}^N\delta_{X_t^j\times X_t^{j,N}}$ is a coupling of $\tilde\mu_t^N$ and $\hat\mu_t^N$.
Subsequently, due to   the fact that $(X_t^i,X_t^{i,N})_{1\le i\le N}$ are identically distributed (see e.g. \cite[p.\ 122--123]{CD2})
by recalling that $(X_0^{i}, X_0^{i,N})_{1\le i\le N}$
  are i.i.d.\ $\mathscr F_0^1$-measurable random variables,
  we derive that
\begin{align*}
\E  |Z_t^{i,N}|
&\le \frac{1}{2}\int_0^t\big((2L_4+L_3^2)\E|Z_s^{i,N}|+L_4\E\mathbb W_1(\mu_s^i,\tilde\mu_s^N)\big)\d s.
\end{align*}
Whence,  \eqref{EE3} follows from  Gronwall's inequality.

(ii) Next, we  prove \eqref{E9}. We firstly verify  that there exists a constant $c_0>0$ such that for all $  i\in\mathbb S_N$
and all $t>0$,
\begin{align}\label{W*}
  \E|X_t^i|\le \big(1+c_0t+\E|X_0^i|\big)\e^{c_0t}.
 \end{align}
Indeed,  applying It\^o's formula to the function $V_1$, defined in \eqref{W0} with $\delta=1$,
and taking  advantage of \eqref{W**} with $\delta=1$,
 we infer  from \eqref{W4}, \eqref{W2--} 
  and $V_1\ge1$ that for some constant    $c_1>0$,
\begin{align*}
 \d V_1(X_t^i)&=\<\nn V_1(X_t^{i}),b(X_t^i,\mu_t^i)\>\d t+\frac{1}{2}\<\nn^2V_1(X_t^{i}), (\si_1^2+\sigma_0^2)I_d+\bar\si(X_t^i)(\bar\si(X_t^i))^*\>_{\rm HS}\d t+\d M_t^i\\
 &\le \frac{1}{V_1(X_t^{i})}\Big(\<X_t^{i},b(X_t^i,\mu_t^i)\>+\frac{1}{2}\big((\si_1^2+\sigma_0^2)d+\|\bar\si(X_t^i)\|_{\rm HS}^2\big)\Big)\d t+ \d M_t^i\\
 &\le c_1\big(1+|X_t^{i}|+\mu_t^i(|\cdot|)\big)\d t+  \d M_t^i,
\end{align*}
 where
\begin{align*}
 \d M_t^i:=\big\<\nn V_1(X_t^{i}),   \si_1\d B_t^{1,i}+\bar\si(X_t^i) \d   B_t^{2,i}+\si_0\d W_t  \big\>.
\end{align*}
 Thus, by invoking the fact that  $$\E\mu_t^i(|\cdot|)=\E\big(\E\big(|X_t^i|\big|\mathscr F_t^W\big)\big)=\E|X_t^i|,$$ we conclude that
 \begin{align*}
 \E|X_t^i|\le 1+\E|X_0^i|+2c_1\int_0^t(1+\E|X_s^i|)\d s.
\end{align*}
Therefore,  \eqref{W*} is attainable by applying   Gronwall's inequality.

With \eqref{W*} at hand, we proceed  to prove \eqref{E9}
.
Since, $\P^0$-almost surely, $\tilde\mu_t^N$
converges weakly to $\mu_t^i$, and
\begin{equation*}
\P^1\Big(\lim_{N\to\8}\tilde\mu_t^N(|\cdot|)=\mu_t^i(|\cdot|)\Big)=1
\end{equation*}
by means  of the law of large numbers, \cite[Theorem 5.5]{CD1} yields $\P^0$-almost surely
\begin{align*}
\P^1\Big(\lim_{N\to\8}\mathbb W_1(\mu_t^i,\tilde\mu_t^N)=0\Big)=1.
\end{align*}
Whereafter, owing to
\begin{align*}
 \mathbb W_1(\mu_t^i,\tilde\mu_t^N)\le  \mu_t^i(|\cdot|)+\tilde\mu_t^N(|\cdot|)
\end{align*}
and  the fact that $X_t^i$ and $X_t^j$ are identically distributed given the filtration $\mathscr F_t^{W}$, the dominated convergence theorem yields that
\begin{align*}
\P^0\Big(\lim_{N\to\8}\E^1\mathbb W_1(\mu_t^i,\tilde\mu_t^N)=0\Big)=1.
\end{align*}
Next, in the  light of
  $$\E^1\mathbb W_1(\mu_t^i,\tilde\mu_t^N) \le 2 \mu_t^i(|\cdot|)\quad \mbox{ and } \quad \E\mathbb W_1(\mu_t^i,\tilde\mu_t^N)=\E^0\big(\E^1\mathbb W_1(\mu_t^i,\tilde\mu_t^N)\big),$$
the verification
 \eqref{W*} and the dominated convergence theorem enable  us to derive \eqref{E9}.
\end{proof}

\subsection{Asymptotic coupling by reflection}\label{section2.2}
 For any $\vv>0$, define the cut-off function $h_\vv$ by
 \begin{equation}\label{E18}
h_\vv(r)=
\begin{cases}
0,\qquad \qquad\qquad\qquad\qquad \qquad  r\in[0,\vv],\\
1-\exp\big((r-\vv)/(r-2\vv)\big),\quad r\in(\vv,2\vv),\\
1,\qquad \qquad\qquad\qquad\qquad\quad\quad   r\ge 2\vv.
\end{cases}
\end{equation}
For any $x\in \R^d$, define $${\bf n}(x)=\frac{x}{|x|}\I_{\{x\neq {\bf0}\}}+(1,0,\cdots,0)^*\I_{\{x= {\bf0}\}},$$
where $a^*$ means the transpose of the   column vector $a\in \R^d$.
For $d\ge1$ and ${\bf x}:=(x_1,\cdots,x_N)\in  (\R^d)^N$, define
\begin{equation}\label{e:ppoodd}
\phi_{d, \vv}({\bf x}):=I_d-2h_\vv(\rho({\bf x})){\bf n}(\varphi( {\bf x}))\otimes {\bf n}(\varphi(\bf x)),
\end{equation}
where $I_d$ means the $d\times d$-identity matrix,
 $\rho:(\R^d)^N\to[0,\8)$, and  $\varphi:(\R^d)^N\to\R^d.$ In particular, for the case $d=1$, we have
 $$\phi_{1, \vv}({\bf x}) =1-2h_\vv(\rho({\bf x})).$$

In order to investigate the issue on uniform-in-time propagation of chaos for the SDE \eqref{W1}, we construct the asymptotic coupling by reflection
 between the non-interacting particle system \eqref{E6} and the corresponding interacting particle system \eqref{E7}. More precisely, we build  the following approximate interacting particle systems: for all $i\in\mathbb S_N$ and $\vv>0,$
\begin{equation}\label{EE0}
\begin{cases}
\d  X_t^i=b(X_t^i,\mu_t^i) \d t+  \si_1\d B_t^{1,i}+\bar\si(X_t^i) \d   B_t^{2,i}+\si_0\d W_t,\\
\d   X_t^{i,N,\vv}=b(X_t^{i,N,\vv},\hat\mu_t^{N,\vv}) \d t +    \si_1\Pi_{\vv}({\bf X}_t^{N}-{\bf X}_t^{N,N,\vv}) \d B_t^{1,i} + \bar\si \big(  X_t^{i,N,\vv}\big)\d   B_t^{2,i},\\
\qquad\qquad+\sigma_0 \Pi_{\vv}({\bf X}_t^{N}-{\bf X}_t^{N,N,\vv})\d W_t.
\end{cases}
\end{equation}
where  $\hat\mu_t^{N,\vv}:=\frac{1}{N}\sum_{j=1}^N\delta_{  X_t^{j,N,\vv}},$ the empirical measure of interacting particles at time $t$,
 \begin{equation*}
{\bf X}_t^{N}: =\big(X_t^{1}, \cdots,X_t^{N}\big), \quad  {\bf X}_t^{N,N,\vv}: =\big(X_t^{1,N,\vv}, \cdots,X_t^{N,N,\vv}\big)
\end{equation*}
and
\begin{equation}\label{e:diff}
\Pi_{\vv}({\bf X}_t^{N}-{\bf X}_t^{N,N,\vv}):= \varphi_{d, \vv}\big({\bf X}_t^{N}-{\bf X}_t^{N,N,\vv}\big).
\end{equation}
We assume $(X_0^{i,N,\varepsilon})_{1\le i\le N}=(X_0^{i,N})_{1\le i\le N}$ for any $\vv>0$,
and recall that  $(X_0^{i}, X_0^{i,N})_{1\le i\le N}$
  are i.i.d.\ $\mathscr F_0^1$-measurable random variables.  Moreover,   we emphasize that      $X_0^i=X_0^{i,N}$ is not required in \eqref{EE0}.

   The main thesis  in this part is presented as follows.

 \begin{proposition}\label{pro1}
 Let $({\bf X}^{N},{\bf X}^{N,N,\vv})_{\vv>0}:=(({\bf X}^{N}_t)_{t\ge0},({\bf X}^{N,N,\vv}_t)_{t\ge0})_{\vv>0}$ be the process determined by the system \eqref{EE0} such that the initial points $(({\bf X}^{N}_0),({\bf X}^{N,N,\vv}_0))_{\vv>0}$ satisfy all the properties mentioned above. 
Under  $({\bf A}_b)$ and $({\bf A}_{\bar\sigma})$,
 $({\bf X}^{N},{\bf X}^{N,N,\vv})_{\vv>0}$ has a  weakly convergent  subsequence
  such that the corresponding weak limit process is the coupling process of ${\bf X}^{N}$ and ${\bf X}^{N,N}$,
 where  ${\bf X}_t^{N,N}: =\big(X_t^{1,N}, \cdots,X_t^{N,N}\big)$ for any $t\ge0.$
\end{proposition}

To prove Proposition \ref{pro1}, we need to show that, for given particle number $N$ and a finite-time horizon $T>0$, the $(\R^{d})^N$-valued process $({\bf X}_t^{N,N,\vv})_{t\ge0} $ owns a uniform moment, which is an ingredient  to illustrate the tightness of $({\bf X}^{N,N,\vv})_{\vv>0}$.

\begin{lemma}
Assume that the Assumptions $({\bf A}_b)$ and $({\bf A}_{\bar\sigma})$ and suppose that $ \E| X_0^{1,N}|<\8.$ Then,  for given $T>0$, there is a constant $C_{T}>0$  
such that for all $\vv>0$ and $N\ge1$,
\begin{equation}\label{E16}
 \E\bigg(\sup_{0\le t\le T}|{\bf X}_t^{N,N,\vv}|\bigg)\le C_{T}N\big(1+\E| X_0^{1,N}|\big).
\end{equation}
\end{lemma}

\begin{proof}
It is easy to see from $h_\vv\in[0,1]$ that for all ${\bf x} \in(\R^d)^N,$
\begin{equation}\label{EWW}
\|\Pi_{ \vv}({\bf x})\|_{\rm HS}^2=
  d+4h_\vv(\rho({\bf x}) )(h_\vv(\rho({\bf x}) )-1) \le d.
\end{equation}
Then,  applying It\^o's formula to $V_1$, introduced in \eqref{W0} with $\delta=1$,
and making use  of    $V_1\ge1$, we deduce from  \eqref{W2}  and \eqref{W4} (see also the arguments below \eqref{W*}) that for some constant $c_1>0,$
\begin{align*}
\d V_1(X_t^{i,N,\vv})&\le\frac{1}{2V_1(X_t^{i,N,\vv})}\big(2  \<X_t^{i,N,\vv}, b(X_t^{i,N,\vv},\hat\mu_t^{N,\vv}) \>+ \|\bar\si(X_t^{i,N,\vv})\|_{\rm HS}^2  +(\si_0^2+\sigma_1^2)d\big)\,\d t+\d M^{i,N,\vv}_t\\
&\le c_1\big(1+| X_t^{i,N,\vv}|+\hat\mu_t^{N,\vv}(|\cdot|)\big)\,\d t+\d M^{i,N,\vv}_t,
\end{align*}
 where
\begin{align*}
\d M^{i,N,\vv}_t: &=  \big\<\nn V_1 (X_t^{i,N,\vv}), \Pi_\vv({\bf X}_t^{N}-{\bf X}_t^{N,N,\vv})  (\sigma_1  \d B_t^{1,i} + \sigma_0\d W_t  )+ \bar\si(X_t^{i,N,\vv})   \d   B_t^{2,i}\big\> .
\end{align*}
For any integer $n\ge1$, define the   stopping time
\begin{equation*}
\tau^{N,\vv}_n=\inf\big\{t\ge0:|{\bf X}_t^{N,N,\vv}|\ge n\big\}.
\end{equation*}
Employing BDG's inequality and taking \eqref{W4} and \eqref{EWW} into consideration
yield  that for some constants $c_2,c_3>0,$
\begin{equation*}\label{E15}
\begin{split}
\gamma_n^{i,N,\vv}(t):&=\E\bigg(\sup_{0\le s\le t\wedge \tau^{N,\vv}_n}| X_s^{i,N,\vv}|\bigg)\\
& \le \E| X_0^{i,N,\vv}|+c_2t +c_2\int_0^t\bigg(\gamma_n^{i,N,\vv}(s)+\frac{1}{N}\sum_{j=1}^N\gamma_n^{j,N,\vv}(s)\bigg)\,\d s\\
&\quad+c_2\E\bigg(\int_0^{t\wedge \tau^{N,\vv}_n}(1+|X^{i,N,\vv}_s|)^2\,\d s\bigg)^{1/2}\\
&\le \E| X_0^{i,N,\vv}|+c_3t +c_3\int_0^t\bigg(\gamma_n^{i,N,\vv}(s)+\frac{1}{N}\sum_{j=1}^N\gamma_n^{j,N,\vv}(s)\bigg)\,\d s+\frac{1}{2}\gamma_n^{i,N,\vv}(t),
\end{split}
\end{equation*} where  in the last inequality we used the fact that $2ab\le \eta^{-1}a^2+\eta b^2$ for all $a,b,\eta>0$.
This obviously   implies that for some constant  $c_4>0$
\begin{equation*}
\frac{1}{N}\sum_{i=1}^N\gamma_n^{i,N,\vv}(t)
 \le c_4\bigg(\E| X_0^{1,N}|+ t + \frac{1}{N}\sum_{j=1}^N\int_0^t \gamma_n^{j,N,\vv}(s) \,\d s \bigg),
\end{equation*}
since $(X_0^{i,N,\vv})_{1\le i\le N}=(X_0^{i,N})_{1\le i\le N}$ are i.i.d. $\mathscr F_0^1$-measurable random variables.
Hence,  by applying  Gronwall's inequality and   Fatou's lemma, there exists a  constant $C_T^*>0 $ such that
\begin{equation*}
\frac{1}{N}\sum_{i=1}^N\E\bigg(\sup_{0\le t\le T }| X_t^{i,N,\vv}|\bigg)\le C_T^*\big(1+\E| X_0^{1,N}|
\big).
\end{equation*}
Noting that
\begin{align*}
\E\bigg(\sup_{0\le t\le T}|{\bf X}_t^{N,N,\vv}|\bigg)\le \sum_{i=1}^N\E\bigg(\sup_{0\le t\le T }| X_t^{i,N,\vv}|\bigg),
\end{align*} the assertion \eqref{E16} follows immediately.
\end{proof}

\begin{lemma}\label{tight}
Assume $({\bf A}_b)$ and $({\bf A}_{\bar\sigma})$, and suppose further $\E|X_0^{1,N}|
<\8$. Then,  the path-valued process $\{{\bf X}^{N,N,\vv}\}_{\vv>0}$ is tight in $\mathscr C_T:=C([0,T];(\R^{d})^N)$ for any given $N\ge1$ and  $T>0$.
\end{lemma}

\begin{proof}
According to \cite[Theorem 1]{Aldous}, for the sake of tightness of
 $\{{\bf X}^{N,N,\vv}\}_{\vv>0}$   in $\mathscr C_T$, it amounts to establishing that
\begin{enumerate}
\item[({\bf i})] for each $t\in[0,T]$, $\{{\bf X}^{N,N,\vv}_t\}_{\vv>0}$ is tight;

\item[({\bf ii})] ${\bf X}^{N,N,\vv}_{\tau_\vv+\delta_\vv}- {\bf X}^{N,N,\vv}_{\tau_\vv}\to 0$ in probability as $\vv\to0$, where, for each $\vv>0$, $ \tau_\vv\in[0,T] $ is a stopping time and  $\delta_\vv\in[0,1]$ is a constant such that  $\delta_\vv\to0$ as $\vv\to0.$
\end{enumerate}
In the sequel, we  aim to verify the two statements above, one by one.

For any $r>0,$ let ${\bf B}_r  =B_r\times B_r\cdots\times B_r\subset (\R^d)^N$, where  $B_r:=\{x\in\R^d:|x|\le r\}$, a compact subset in $\R^d.$  Let $B_r^c$ and ${\bf B}_r^c$
be the respective complements of $B_r$ and ${\bf B}_r$.
 By the Chebyshev inequality, in addition to \eqref{E16},
  we find  that   for any  $t\in[0,T] $ and $R>0$,
\begin{equation*}
\begin{split}
\P\big({\bf X}_t^{N,N,\vv}\in {\bf B}_R^c\big)&\le (2^N-1)\max_{i\in\mathbb S_N}\P\big(X_t^{i,N,\vv}\in B_R^c\big)\\
&\le  \frac{1}{R}(2^N-1)C_{T}N\big(1+\E| X_0^{1,N}|
\big).
\end{split}
\end{equation*}
By virtue of the estimate above,   the statement ({\bf i}) is valid right now.

  For any $\beta>0$, it is easy to notice that
\begin{align*}
\P\big(\big|{\bf X}^{N,N,\vv}_{\tau_\vv+\delta_\vv}- {\bf X}^{N,N,\vv}_{\tau_\vv}\big|\ge \beta\big)
&\le \sum_{i=1}^N \Bigg(\P\bigg(\int_{\tau_\vv}^{\tau_\vv+\delta_\vv}\big|b(  X_s^{i,N,\vv},\hat{ \mu}_s^{N,\vv})\big|\,\d s\ge \frac{\beta}{4N}\bigg) \\
&\qquad\quad+\P\bigg( |\si_1| \bigg|\int_{\tau_\vv}^{\tau_\vv+\delta_\vv}\Pi_\vv({\bf X}_s^{N}-{\bf X}_s^{N,N,\vv})\,\d B_s^{1,i}\bigg|\ge \frac{\beta}{4N}\bigg)\\
&\qquad\quad + \P\bigg(  |\sigma_0| \bigg|\int_{\tau_\vv}^{\tau_\vv+\delta_\vv}\Pi_\vv({\bf X}_s^{N}-{\bf X}_s^{N,N,\vv})\,\d W_s \bigg|\ge \frac{\beta}{4N}\bigg)\\
&\qquad\quad+ \P\bigg( \bigg|\int_{\tau_\vv}^{\tau_\vv+\delta_\vv}\bar\si(X_s^{i,N,\vv})\,\d B_s^{2,i}\bigg|\ge \frac{\beta}{4N}\bigg)\Bigg) \\ &=:\sum_{i=1}^N\sum_{j=1}^4 \Gamma^{j,\vv}_i.
\end{align*}
In the event of  $\sigma_1,\sigma_0=0$, then $\Gamma^{2,\vv}_i=\Gamma^{3,\vv}_i=0$ holds true trivially  so we shall prescribe  $\sigma_1\neq 0$ and $\sigma_0\neq0$ in the subsequent  analysis.
For any $R_0>0$,    applying Chebyshev's inequality followed by \eqref{E16} yields that
\begin{equation*}
\P\bigg(\sup_{0\le t\le T+1}|{\bf X}_t^{N,N,\vv}|\ge R_0\bigg)\le \frac{1}{R_0 }C_{T+1}N\big(1+\E| X_0^{1,N}|
 \big).
\end{equation*}
Hence, for any $\vv_0>0,$ we can take $R_0^*=R_0^*(\vv_0)>0$ large enough so that
\begin{align}\label{E2}
\P\bigg(\sup_{0\le t\le T+1}|{\bf X}_t^{N,N,\vv}|\ge R_0^*\bigg)\le \vv_0.
\end{align}
For $R_0^*>0$ stipulated above,  we define  the stopping time
\begin{equation*}
\tau_0^{N,\vv}=\inf\big\{t\ge0: |{\bf X}_t^{N,N,\vv}|\ge R_0^*\big\}.
\end{equation*}
Whereafter, the term $\Gamma_i^{1,\vv}$ can be estimated as below:
\begin{align*}
 \Gamma^{1,\vv}_i&\le\P\bigg(\int_{\tau_\vv}^{\tau_\vv+\delta_\vv}\big|b(  X_s^{i,N,\vv},\hat{ \mu}_s^{N,\vv})-b(  X_s^{i,N,\vv},\delta_0)\big|\,\d s\ge \frac{\beta}{8N}\bigg)\\
 &\quad+\P\bigg(\int_{\tau_\vv}^{\tau_\vv+\delta_\vv}\big|b(  X_s^{i,N,\vv},\delta_0)\big|\,\d s\ge \frac{\beta}{8N}\bigg)\\
 &\le \P\bigg( \int_{\tau_\vv}^{\tau_\vv+\delta_\vv}\mathbb W_1(\hat{ \mu}_s^{N,\vv},\delta_0) \,\d s\ge \frac{\beta}{8NL_2}\bigg)+ \P\big( \tau_0^{N,\vv}\le T+1\big)\\
 &\quad+\P\bigg(\int_{\tau_\vv}^{\tau_\vv+\delta_\vv}\big|b(  X_s^{i,N,\vv},\delta_0)\big|\,\d s\ge \frac{\beta}{8N},\tau_0^{N,\vv}> T+1\bigg)\\
 &\le  \P\bigg( \frac{1}{N}\sum_{j=1}^N\int_{\tau_\vv}^{\tau_\vv+\delta_\vv}|X_s^{j,N,\vv}| \,\d s\ge \frac{\beta}{8NL_2}\bigg)+ \P\bigg(\sup_{0\le t\le T+1}|{\bf X}_t^{N,N,\vv}|\ge R_0^*\bigg)\\
 &\quad+\P\bigg(\int_{\tau_\vv}^{\tau_\vv+\delta_\vv}\I_{[0,\tau_0^{N,\vv}]}(s)\big|b(  X_s^{i,N,\vv},\delta_0 )\big|\,\d s\ge \frac{\beta}{8N} \bigg),
\end{align*}
where   the second inequality  holds true due to \eqref{W3}.
As a consequence, by taking  \eqref{E16} and \eqref{E2} into account and retrospecting  that $b(\cdot,\delta_0)$ is continuous and locally bounded on $\R^d$ (see the Assumption $({\bf A}_b)$) and $\lim_{\vv\downarrow0}\delta_\vv=0$, we conclude that $ \lim_{\vv\downarrow0}\Gamma^{1,\vv}_i=0 $.

On the one hand,
by applying  Chebyshev's inequality and It\^o's isometry, along with \eqref{W4} and \eqref{EWW}, it follows  that
\begin{align*}
 \Gamma^{2,\vv}_i+\Gamma^{3,\vv}_i&\le \frac{16N^2 }{\beta^2}\bigg((\si_0 ^2+ \sigma_1 ^2)\E\bigg(\int_{\tau_\vv}^{\tau_\vv+\delta_\vv}\big\|\Pi_\vv({\bf X}_s^{N}-{\bf X}_s^{N,N,\vv})\big\|_{\rm HS}^2 \d s\bigg)\bigg)\\
&\le
\frac{16N^2 }{\beta^2}(  \si_0^2+\sigma_1^2 )d\delta_\vv.
\end{align*}
On the other hand,  in terms of \cite[Lemma 2.3]{Frie}, for $\vv_0>0$ given in \eqref{E2}, we find that
\begin{align*}
 \Gamma^{4,\vv}_i&\le
 \vv_0+\P\bigg( \bigg|\int_{\tau_\vv}^{\tau_\vv+\delta_\vv}\big\|\bar\si(X_s^{i,N,\vv})\big\|^2_{\rm HS} \d  s \bigg|\ge \frac{\beta^2\vv_0}{16N^2}\bigg)\\
 &\le \vv_0+\P\bigg(\sup_{0\le t\le T+1}|{\bf X}_t^{N,N,\vv}|\ge R_0^*\bigg)+\P\bigg( \bigg|\int_{\tau_\vv}^{\tau_\vv+\delta_\vv}\I_{[0,\tau_0^{N,\vv}]}(s)\big\|\bar\si(X_s^{i,N,\vv})\big\|^2_{\rm HS} \d  s \bigg|\ge \frac{\beta^2\vv_0}{16N^2}\bigg)\\
 &\le 2\vv_0+\P\bigg( \bigg|\int_{\tau_\vv}^{\tau_\vv+\delta_\vv}\I_{[0,\tau_0^{N,\vv}]}(s)\big\|\bar\si(X_s^{i,N,\vv})\big\|^2_{\rm HS} \d  s \bigg|\ge \frac{\beta^2\vv_0}{16N^2}\bigg),
\end{align*}
where the second inequality is obtained by following the line to deal with the term $ \Gamma^{1,\vv}_i$, and the last display is owing to \eqref{E2}.
Consequently, with the aid of     $\lim_{\vv\downarrow0}\delta_\vv=0$ and the Lipschitz property of $\bar\si$ (so it is continuous  and locally bounded on $\R^d$),   the conclusion
$
\sum_{j=2}^4\lim_{\vv\downarrow0}\Gamma^{j,\vv}_i=0
$
is reachable.
At length,  the statement ({\bf ii}) is verifiable by recalling $ \lim_{\vv\downarrow0}\Gamma^{1,\vv}_i=0 $.
\end{proof}

With Lemma \ref{tight} at hand, we intend to complete the
\begin{proof}[Proof of Proposition $\ref{pro1}$]
Let $\mathscr C_\8= C([0,\8);(\R^{d})^N)$ be the collection of  continuous functions $\psi:[0,\8)\to (\R^{d})^N$.
Define the projection operator  $\pi: \mathscr C_\8\to (\R^{d})^N$   by $\pi_t\psi=\psi(t)$ for $\psi\in \mathscr C_\8$ and $t\ge0,$
 and $\mathcal F_t=\si(\pi_s:s\le t)$ by the $\sigma$-algebra on $\mathscr C_\8$ induced by the projections $\pi_s$  for $s\in[0,t].$

With the help of  Lemma \ref{tight}, the Prohorov theorem  yields  that $({\bf X}^{N},{\bf X}^{N,N,\vv})_{\vv>0}$ has a weakly convergent  subsequence $({\bf X}^{N},{\bf X}^{N,N,\vv_l})_{l\ge 0} $
with the associated weak limit $({\bf X}^{N},\tilde{{\bf X}}^{N,N}) $, where $(\vv_l)_{l\ge0}$ is a sequence such that $\lim_{l\to \infty}\vv_l=0.$
To demonstrate  that
$({\bf X}^{N},\tilde{{\bf X}}^{N,N}) $ is indeed a coupling process of ${\bf X}^{N}$ and ${\bf X}^{N,N}$, it is sufficient to verify that $
\mathscr L_{\tilde{{\bf X}}^{N,N}}=\mathscr L_{{\bf X}^{N,N}},
$ where $\mathscr L_{\tilde{{\bf X}}^{N,N}}$ and $\mathscr L_{{\bf X} ^{N,N}}$ are the infinitesimal generators of $(\tilde{\bf X}_t^{N,N})_{t\ge0}$ and  $({\bf X}_t^{N,N})_{t\ge0}$, respectively. In particular, we have
for $f\in C_c^2((\R^d)^N)$,
\begin{equation*}
\begin{split}
\big(\mathscr L_{{\bf X} ^{N,N}}f\big)({\bf x}) =\sum_{i=1}^N\bigg(&\<\nn_if( {\bf x} ),  b(x^i,\hat \mu_{{\bf x}}^N)\>  +\frac{1}{2}\si_1^2\,\mbox{trace} (\nn^2_{ii}f( {\bf x} ))+\frac{1}{2}\<\nn^2_{ii}f( {\bf x} ),\bar\si(x^i)(\bar\si(x^i))^*\>_{\rm HS}\\
&+\frac{1}{2}\si_0^2\sum_{j=1}^N\mbox{trace}\big(\nn_{ij}^2f( {\bf x} )\big)\bigg),
\end{split}
\end{equation*}
where  $\hat\mu_{\bf x}^N:=\frac{1}{N}\sum_{j=1}^N\delta_{x_j}$.

To realize this goal,   we define for any $f\in C^2_c((\R^{d})^N)$,
\begin{align*}
M_t^{N,f}=f(\tilde{{\bf X}}^{N,N}_t)-f(\tilde{{\bf X}}^{N,N}_0)-\int_0^t(\mathscr L_{{\bf X} ^{N,N}}f)(\tilde{{\bf X}}^{N,N}_s)\,\d s.
\end{align*}
For any $f\in C^2_c((\R^{d})^N)$,
provided that  $(M_t^{N,f})_{t\ge0}$ is a martingale with respect to $(\mathcal F_t)_{t\ge0}$, i.e., for any $t\ge s\ge0$ and $\mathcal F_s$-measurable bounded continuous functional $F:\mathscr C_\8\to\R$,
\begin{equation}\label{E21}
\E\big(M_t^{N,f}F(\tilde{{\bf X}}^{N,N})\big)=\E\big(M_s^{N,f}F(\tilde{{\bf X}}^{N,N})\big),
\end{equation}
via the weak uniqueness of \eqref{E7},
the assertion   $\mathscr L_{\tilde{{\bf X}}^{N,N}}=\mathscr L_{{\bf X}^{N,N}}$ is available, and so $({\bf X}^{N},\tilde{{\bf X}}^{N,N}) $ is   a coupling process of ${\bf X}^{N}$ and ${\bf X}^{N,N}$.

 Below, we intend to prove the assertion \eqref{E21}.
For fixed ${\bf x}\in(\R^{d})^N$, let $\mathscr L^{N,\vv}_{{\bf x}}$ be the infinitesimal generator of $({\bf X}_t^{N,N,\vv})_{t\ge0}$ provided that  the Markov process $( {\bf X}_t^{N,N })_{t\ge0}$ is known in advance.

For any  $f\in C^2((\R^{d})^N)$, ${\bf y}\in(\R^{d})^N$, and given ${\bf x}\in(\R^{d})^N$, we observe that
\begin{equation}\label{E3}
\begin{split}
\big(\mathscr L^{N,\vv}_{{\bf x}}f\big)({\bf y})&=\big(\mathscr L_{{\bf X} ^{N,N}}f\big)({\bf y})- 2\bigg(\si_1^2\sum_{i=1}^N  \<\nn_{ii}^2 f({\bf y}),
{\bf n}(\varphi({\bf x-y}))\otimes {\bf n}(\varphi({\bf x-y}))\>_{\rm HS}  \\
&\qquad\qquad\qquad\qquad\qquad+\si_0^2\sum_{i,j=1}^N\<\nn^2_{ij}f({\bf y}),{\bf n}(\varphi({\bf x-y}))\otimes {\bf n}(\varphi({\bf x-y}))\>_{\rm HS}\bigg)\\
&\qquad\qquad\qquad\qquad\qquad\times h_\vv(\rho({\bf x-y}))\big(1-h_\vv(\rho({\bf x-y}))\big)\\
&=:\big(\mathscr L_{{\bf X} ^{N,N}}f\big)({\bf y})-\big(\mathscr L^{N,\vv,*}_{{\bf x}}f\big)({\bf y}).
\end{split}
\end{equation}

By It\^o's formula, for any $f\in C_c^2(\R^{N})$ and  $t\ge 0,$
\begin{align*}
M_t^{N,f,\vv_l}:=f({\bf X}^{N,N,\vv_l}_t)-f({\bf X}^{N,N,\vv_l}_0)-\int_0^t\big(\mathscr L^{N,\vv_l}_{{\bf X}^{N}_s}f\big)({\bf X}^{N,N,\vv_l}_s)\,\d s
\end{align*}
is a martingale with respect to $(\mathcal F_t)_{t\ge0}$. Therefore, for any $t\ge s\ge0$ and $\mathcal F_s$-measurable bounded continuous functional $F:\mathscr C_\8\to\R$, we obviously have
\begin{equation}\label{E22}
\E\big(M_t^{N,f,\vv_l}F( {\bf X}^{N,N,\vv_l})\big)=\E\big(M_s^{N,f,\vv_l}F( {\bf X}^{N,N,\vv_l})\big).
\end{equation}
Next, owing to \eqref{E3}, $M_t^{N,f,\vv_l}$ can be rewritten as below
\begin{align*}
M_t^{N,f,\vv_l}&=f({\bf X}^{N,N,\vv_l}_t)-f({\bf X}^{N,N,\vv_l}_0)-\int_0^t(\mathscr L_{{\bf X} ^{N,N}}f)({\bf X}^{N,N,\vv_l}_s)\,\d s
+\int_0^t \big(\mathscr L^{N,\vv_l,*}_{{\bf X}^{N}_s}f\big)({\bf X}^{N,N,\vv_l}_s)\,\d s
\end{align*}
Whence, the  assertion \eqref{E21} is reachable by applying \eqref{E22} and \cite[Lemma A.2]{Suzuki} as well as the dominated convergence theorem, and also taking  the fact that
$$\lim_{\vv\to 0}\big(\mathscr L^{N,\vv,*}_{{\bf x}}f\big)({\bf y})=0$$
into consideration,
thanks to
\begin{align*}
\lim_{\vv\downarrow0}\big(h_\vv(r)(1-h_\vv(r))\big)=\lim_{\vv\downarrow0}h_\vv(r)\lim_{\vv\downarrow0}(1-h_\vv(r))=\I_{\{r\neq 0\}}(1-\I_{\{r\neq0\}})=0,\quad r\ge0.
\end{align*}
The proof is complete.
\end{proof}

Before ending of  this section, we make a  comment  on the asymptotic coupling by reflection constructed in \eqref{EE0}.
 \begin{remark}\label{remark:coup}
 In terms of \eqref{e:diff}, the asymptotic reflection matrix
 $\phi_{d,\vv}$ embodies   the
information  concerned with all particles, which are common for each single particle. Intuitively, such construction is reasonable since we design the coupling for the system \eqref{EE0} determined  by all   particles rather than the single particle. Indeed, from the argument above, one can see that
 such an observation plays an extremely important role in verifying that the weak limit process of $({\bf X}_t^{N}, {\bf X}_t^{N,N,\vv})_{\vv>0}$ is the coupling process we expect.
Once
 $\phi_{d,\vv}$ contains only partial information of all particles, due to the involvement of the common noise, it is impossible to  examine that the weak limit process of $({\bf X}_t^{N}, {\bf X}_t^{N,N,\vv})_{\vv>0}$ is the coupling process
by a close inspection of the proof for Proposition \ref{pro1}.
For the case $d\ge2$, if we take $\rho({\bf x})=|x_i|$ and $\varphi({\bf x})=x_i/|x_i|$ for $i\in\mathbb S_N$, then
the  intractable term, for $x_i,y_i\in\R^d,$
 \begin{align*}
 &\quad h_\vv(|x_i-y_i|) {\bf n}(x_i-y_i)({\bf n}(x_i-y_i))^*  +h_\vv(|x_j-y_j|) {\bf n}(x_j-y_j)({\bf n}(x_j-y_j))^* \\
&\qquad -2h_\vv(|x_i-y_i|)h_\vv(|x_j-y_j|)\<{\bf n}(x_i-y_i),{\bf n}(x_j-y_j)\>  {\bf n}(x_i-y_i)({\bf n}(x_j-y_j))^*
 \end{align*}
 appears naturally in the infinitesimal generator of $({\bf X}_t^{N,N,\vv})_{t\ge0}$. However, as $\vv\to0$, the preceding term  need not converge to zero.
  This  definitely  brings essential difficulties to identify the
 weak limit process of $({\bf X}^{N,N,\vv})_{\vv>0}$.
\end{remark}

\section{Proof of Theorem \ref{thm1}}\label{sec3}

Our goal in this section is to complete the proof of Theorem \ref{thm1}.
In particular, we are only concerned with the case $d=1$.
 To this end, there are a series of  preparations  to be carried out.

\ \

The non-interacting particle system corresponding to \eqref{E1} is  governed by the following SDEs:
for each $i\in\mathbb S_N,$
\begin{equation}\label{E11}
\d X_t^i= b(X_t^i,\mu_t^i) \,\d t+  \si(X_t^i) \,\d   B_t^i+\si_0\,\d W_t,
\end{equation}
where $\mu_t^i:=\mathscr L_{X_t^i|\mathscr F_t^W}$,
$(B^i)_{i\in\mathbb S_N}:=((B^i_t)_{t\ge0})_{i\in \mathbb S_N}$
are mutually independent $1$-dimensional Brownian motions on $(\OO^1,\mathscr F^1,(\mathscr F^1_t)_{t\ge0}, \P^1)$, and  $(X_0^{i})_{1\le i\le d}$ are i.i.d.\ $\mathscr F_0^1$-measurable random variables. According to \cite[Proposition 2.11]{CD2}, for any $T>0$ and  $i\in\mathbb S_N,$
\begin{equation*}
 \P^0\big( \mu_t^i= \mu_t^1\quad \mbox{for all } t\in[0,T]\big)=1
 \end{equation*}
 so that we can write $\mu_t=\mu_t^i$ for all $i\in\mathbb S_N.$
 Moreover, as  shown   in \cite[(2.4)]{CD2}, $(\mu_t)_{t>0}$ solves the nonlinear stochastic FPE:
 \begin{equation}\label{E12}
 \d \mu_t=-\partial_x (b(\cdot,\mu_t)\mu_t)\,\d t+\frac{1}{2}\partial_{xx}^2\big( (\si^2(\cdot)+\si^2 _0 )\mu_t\big)\,\d t-\partial_x\big((\sigma_0\d W_t)\mu_t\big).
 \end{equation}
The preceding  SPDE is understood in the weak sense;  namely, for any   test function $f\in C_c^2(\R)$,
\begin{equation*}
 \d\mu_t(f)=  \mu_t\big(f'(\cdot)b(\cdot,\mu_t)\big)\,\d t+\frac{1}{2} \mu_t\big((\si(\cdot)^2+\si_0^2)f''(\cdot)\big)\,\d t + \mu_t\big(f'(\cdot)\si_0\,\d W_t \big).
 \end{equation*}

 To expound that   the idiosyncratic noise might   make  contributions to ergodicity of the measure-valued Markov process $(\mu_t)_{t>0}$ solving \eqref{E12}, we need to decompose the   idiosyncratic noise part so that an asymptotic coupling by reflection can be constructed.
  Due to $\kk_{\sigma,1}\le \si(x)^2$ (see the Assumption $({\bf H}_\sigma)$), there exists a constant $\alpha>0$ such that $\inf_{x\in\R}\bar{\si}_\alpha(x)>0$, where
\begin{align}\label{EE7}
\bar{\si}^2_\alpha(x):=\si^2(x)-\alpha \kk_{\si,1}.
 \end{align}
 Subsequently, we consider the stochastic particle system:
\begin{equation*}
\d \bar X_t^i= b(\bar X_t^i,\bar\mu_t^i) \,\d t+  \ss{\aa\kk_{\sigma,1}}\,\d B_t^{1,i}+\bar\si_\alpha(\bar X_t^i) \,\d   B_t^{2,i}+\si_0\,\d W_t,
\end{equation*}
where $\bar \mu_t^i:=\mathscr L_{\bar X_t^i|\mathscr F_t^W}$, $(B^{1,i})_{i\in\mathbb S_N}:=((B^{1,i}_t)_{t\ge0})_{i\in \mathbb S_N}$ and $(B^{2,i})_{i\in\mathbb S_N}:=((B^{2,i}_t)_{t\ge0})_{i\in \mathbb S_N}$
are mutually independent $1$-dimensional Brownian motions on $(\OO^1,\mathscr F^1,(\mathscr F^1_t)_{t\ge0}, \P^1)$, and $(\bar X_0^{i})_{1\le i\le d}$ are i.i.d.\ $\mathscr F_0^1$-measurable random variables.
Once more, applying \cite[Proposition 2.11]{CD2}, we find that for   any $T>0$ and  $i\in\mathbb S_N,$
\begin{equation*}
 \P^0\big( \bar\mu_t^i= \bar\mu_t^1\quad \mbox{ for all } t\in[0,T]\big)=1
 \end{equation*}
 so we can also write $\bar\mu_t=\bar\mu_t^i$ for all $i\in\mathbb S_N.$
Satisfactorily, by noting   $\si(x)^2=\bar\si_\alpha(x)^2+\aa\kk_{\si,1}$,
$(\bar\mu_t)_{t\ge0}$ also solves the SPDE \eqref{E12}. Therefore, to tackle   ergodicity of the measure-valued process $(\mu_t)_{t>0}$, it is sufficient to work on the McKean-Vlasov SDE with common noise:
\begin{equation}\label{EE6}
\d   \bar X_t= b(\bar X_t, \bar\mu_t) \,\d t+  \ss{\aa\kk_{\sigma,1}}\,\d B_t^{1}+\bar\si_\alpha(\bar X_t) \,\d   B_t^{2}+\si_0\,\d W_t.
\end{equation}
The previous  interpretations roughly explain why we deal with the McKean-Vlasov SDE with common noise formulated in the form of \eqref{W1}.

\ \

As analysis above, in this section, we still take  the SDE \eqref{W1} with $d=1$ as our research object.
 Besides the Assumptions  $({\bf H}_{b,1})$ and $({\bf H}_{b,2})$ presented in the Introduction section,   we shall assume that
\begin{enumerate}\it
\item[{\rm(${\bf H}_{\bar\sigma}'$)}] there exists a constant $L_{\bar\si}>0$ such that
\begin{align*}
|\bar\si(x)-\bar\si(y)|\le L_{\bar\si}\big(1\wedge |x-y|\big),\quad x,y\in\R.
\end{align*}
\end{enumerate}
Note that under  the Assumptions  $({\bf H}_{b,1})$, $({\bf H}_{b,2})$ and $({\bf H}_{\bar\sigma}')$, the assumptions $({\bf A}_b)$ and $({\bf A}_{\bar\sigma})$ hold trivially, and so the results in the Section 2 are applicable.

Let $(X_t)_{t\ge0}$ be the unique strong solution to  the SDE \eqref{W1} so that the distribution of $X_0$ is given by $\mu$. For every $t>0$, let $\mu_t =\mathscr L_{X_t|\mathscr F_t^{W}}$ be the regular
conditional distribution of $X_t$ given the $\sigma$-algebra $\mathscr F_t^{W}$.
Before  the proof of Theorem \ref{thm1}, we shall show that $(\mu_t)_{t>0}$,
associated with \eqref{W1}, is exponentially decay  under $\mathcal W_1$-Wasserstein distance.
In detail, we attempt  to  attest  the following statement.

\begin{theorem}\label{thm2}
Assume $({\bf H}_{b,1})$, $({\bf H}_{b,2})$ and $({\bf H}_{\bar\sigma}')$. Then, there exists a constant $\lambda_3^*>0$ such that for
all $\lambda_3\in[0,\lambda_3^*]$ so that there exist constants $C,\lambda>0$ so that for all $t>0$ and $\mu,\nu\in \mathscr P_1(\R)$,
\begin{equation}\label{E20}
\mathcal W_1(\mu P_t,\nu P_t)\le C\e^{-\lambda t}\mathbb W_1(\mu,\nu).
\end{equation}
\end{theorem}

Prior to the commencement on the proof of Theorem \ref{thm2}, some additional work need to be accomplished. In the first place, for each $i\in\mathbb S_N$, $(X_t^i)_{t\ge0}$ has a   finite moment  in the infinite-time horizon, which is stated as below.

 \begin{lemma}\label{lem1}
Assume $({\bf H}_{b,1})$ with  $\lambda_2>2\lambda_3$ and $({\bf H}_{\bar\sigma}')$. Then,
 there is a constant $C_0>0$ such that
\begin{align}\label{E8}
\sup_{i\in\mathbb S_N}\sup_{t\ge0}\E|X_t^i| \le \E|X_0^1| + C_0
\end{align}
in case that $(X_0^i)_{i\in\mathbb S_N}$ are i.i.d. $\mathscr F_0^1$-measurable random variables such that $\E|X_0^1|<\8$.
\end{lemma}

\begin{proof}
According to  \eqref{E4}, \eqref{EE1} and (${\bf H}_{\bar\sigma}'$), for all $x\in\R$ and $\mu\in \mathscr P_1(\R)$,
\begin{equation*}
\begin{split}
2x b(x,\mu)+\bar\si(x)^2 &= 2x(b(x,\mu)-b(0,\delta_0)) +2x b(0,\delta_0) +\bar\si(x)^2\\
&= 2x(b(x,\mu)-b(0,\mu))+2x(b(0,\mu)-b(0,\delta_0)) +2x b(0,\delta_0) +\bar\si(x)^2\\
&\le  -\lambda_2|x|^2+2\lambda_3\mathbb W_1(\mu,\delta_0)|x|+2|b(0,\delta_{0})||x|+(\lambda_1+\lambda_2)\ell_0+2\big(L_{\bar\sigma}^2+ \bar\si(0)^2\big).
\end{split}
\end{equation*}
Then,
 applying  It\^o's formula to $V_1$, defined in \eqref{W0} with $\delta=1$,
 yields that
\begin{equation*}\label{E14}
\begin{split}
\d \big(\e^{ \lambda^*t}V_1(X_t^i)\big)&\le\e^{ \lambda^*t}\bigg( \lambda^*V_1(X_t^i) +\frac{1}{2V_1(X_t^i)}\big( 2 X_t^i b(X_t^i, \mu_t^i)   +\bar\si(X_t^i)^2  + \si_1^2+ \si_0^2\big) \bigg) \d t +\d M_t^i\\
&\le -  \lambda_3\e^{ \lambda^*t} \big(V_1(  X_t^i)-\mu_t^i(V_1)\big)\,\d t  +C_0^*\e^{ \lambda^*t}\,\d t+ \d M_t^i
\end{split}
\end{equation*}
for some $C_0^*>0$ and
 some martingale $(M_t^i)_{t\ge0}$, where  $\lambda^*:=\frac{1}{2}(\lambda_2-2\lambda_3)$.
Since
$
\E\,\mu_t^i(V_1)=\E V_1(X_t^i),
$
and  $(X_0^i)_{i\in\mathbb S_N}$ are i.i.d. $\mathscr F_0^1$-measurable random variables,
 we derive  that
\begin{equation*}
\E V_1(X_t^i)
 \le \E V_1( X_0^1)+ C_0^*/\lambda^*.
\end{equation*}
 This subsequently implies   the desired  assertion \eqref{E8}.
\end{proof}

The following exposition demonstrates the asymptotic  conditional propagation of chaos in the long-time horizon for the conditional McKean-Vlasov SDE \eqref{W1}, which is extremely important  on treating  ergodicity of the measure-valued process $(\mu_t)_{t>0}$. Recall that we focus on the setting  $d=1$. Let $\rho({\bf x})= \|{\bf x}\|_{1}:=\frac{1}{N}\sum_{j=1}^N|x_j|$ in \eqref{e:ppoodd}. In particular, we obtain from \eqref{e:ppoodd} that
\begin{align*}
\varphi_{1, \vv}({\bf x}) =1-2h_\vv( \|{\bf x}\|_{1}),
\quad {\bf x}\in \R^N.
\end{align*}
With the function $\varphi_{1, \vv}({\bf x})$ above, we consider the system $({\bf X}_t^{N}, {\bf X}_t^{N,N,\vv})_{\vv>0}$ determined  by  \eqref{EE0}, where $(X_0^{i,N,\varepsilon})_{1\le i\le N}=(X_0^{i,N})_{1\le i\le N}$ for any $\vv>0$, and $(X_0^{i}, X_0^{i,N})_{1\le i\le N}$
  are i.i.d.\ $\mathscr F_0^1$-measurable random variables.

We have the following statement.

\begin{proposition}\label{pro3}
Assume   $({\bf H}_{b,1})$ with  $\lambda_2>2\lambda_3$, $({\bf H}_{b,2})$  and $({\bf H}_{\bar \sigma}')$, and suppose
\begin{align}\label{E28}
\lambda_0^*:=\frac{c_2 \ell_0   }{1-\e^{-c_1\ell_0}+c_2\ell_0}\big( \lambda_1 \wedge (\lambda_2/2) \big)-\Big(1+
\frac{c_1}{c_2}\Big)\lambda_3 >0,
\end{align}
where
\begin{equation}\label{E0}
c_1:=\frac{\lambda_1\ell_0}{\sigma_0^2+\si_1^2},\qquad c_2:=c_1\e^{-c_1\ell_0}.
\end{equation}
Then,
there exists a  constant $C_0>0$ such that
\begin{equation}\label{E24}
\frac{1}{N}\sum_{i=1}^N\E | Z_t^{i,N,\vv}| \le  \e^{-\lambda_0^*t}\frac{C_0}{N}\sum_{j=1}^N\E | Z_0^{i,N,\vv}|+C_0\bigg(\frac{1}{N}\big(1+\E |X_0^1|\big)+\varphi(N)+   \vv \bigg),
\end{equation}
where $Z_t^{i,N,\vv}:=X_t^i-X_t^{i,N,\vv}$.
\end{proposition}

Below, we make a remark on the decay rate $\lambda_0^*$ given in \eqref{E28}.
\begin{remark}\label{Remark--ss} Note that \begin{align*}\lambda_0^*:=&\frac{c_1 \e^{-c_1\ell_0} \ell_0   }{1-\e^{-c_1\ell_0}+c_2\ell_0}\big( \lambda_1  \wedge (\lambda_2/2) \big)-\Big(1+
\frac{c_1}{c_2}\Big)\lambda_3\\
=&\left(\ell_0+\frac{\e^{c_1\ell_0}-1}{c_1}\right)^{-1}\big(  \lambda_1  \wedge (\lambda_2/2) \big)-(1+\e ^{c_1\ell_0})\lambda_3.\end{align*} It follows from the increasing property of the functions $c_1\mapsto \frac{\e^{c_1\ell_0}-1}{c_1}$ and $c_1\mapsto \e ^{c_1\ell_0}$ on $(0,\infty)$ that the non-degenerate property of $\sigma_0$ and $\si_1$ will facilitate the exponential contractivity of the difference process $(Z_t^{i,N})_{t\ge0}$.
On the other hand, it is easily seen from the expression above for the constant $\lambda_0^*$ that the smaller $\ell_0$ or the larger $\lambda_2$ implies the faster convergence rate in \eqref{E24}. Similarly, the larger $\lambda_1$ indicates the slower convergence rate in \eqref{E24}.
 \end{remark}

\begin{proof}[Proof of Proposition $\ref{pro3}$]
By the It\^o-Tanaka formula (see e.g. \cite[Theorem 142]{Situ}), besides  $(|x|)''=0$ for $x\neq 0,$
we derive that
\begin{equation}\label{E27}
\begin{split}
\d |Z_t^{i,N,\vv}|&=\frac{Z_t^{i,N,\vv}}{|Z_t^{i,N,\vv}|}\I_{\{Z_t^{i,N,\vv}\neq 0\}}\big(b(X_t^i,\mu_t^i)-b(X_t^{i,N,\vv},\hat\mu_t^{N,\vv})\big)\,\d t+\d  \tilde M_t^{i,N,\vv}\\
&\le \frac{Z_t^{i,N,\vv}}{|Z_t^{i,N,\vv}|}\I_{\{Z_t^{i,N,\vv}\neq0\}}\big(b(X_t^i,\tilde\mu_t^N)-b(X_t^{i,N,\vv},\hat\mu_t^{N,\vv})\big)\,\d t  \\
&\quad+ \big|b(X_t^i,\tilde{\mu}_t^{N,i})-b(X_t^i,\tilde\mu_t^{N})\big|\,\d t+ \big|b(X_t^i,\mu_t^i)-b(X_t^i,\tilde{\mu}_t^{N,i})\big|\,\d t +
\d  \tilde M_t^{i,N,\vv},
\end{split}
\end{equation}
where
\begin{align*}
\d  \tilde M_t^{i,N,\vv}:&=\mbox{sgn}\big(Z_t^{i,N,\vv}\big)\Big(2h_\vv\big(\|{\bf Z}_t^{N,N,\vv}\|_1\big)\big( \sigma_1  \d    B_t^{1,i}   + \sigma_0  \d W_t \big) + \big( \bar\si( X_t^i)- \bar\si(  X_t^{i,N,\vv})\big)\,\d   B_t^{2,i}\Big)\\
\end{align*}
with ${\bf Z}_t^{N,N,\vv}:=(Z_t^{1,N,\vv},\cdots, Z_t^{N,N,\vv})$,
and
\begin{align*}
\tilde\mu_t^N:=\frac{1}{N}\sum_{j=1}^N\delta_{X_t^j},\quad \tilde{\mu}_t^{N,i}:=\frac{1}{N-1}\sum_{j=1:j\neq i}^N\delta_{X_t^j}.
\end{align*}

By means of \eqref{E4} and \eqref{EE1}, it follows that
\begin{equation}\label{E25}
\begin{split}
 &\frac{Z_t^{i,N,\vv}}{|Z_t^{i,N,\vv}|}\I_{\{Z_t^{i,N,\vv}\neq0\}}\big(b(X_t^i,\tilde\mu_t^N)-b(X_t^{i,N,\vv},\hat\mu_t^{N,\vv})\big)\\
 &= \frac{Z_t^{i,N,\vv}}{|Z_t^{i,N,\vv}|}\I_{\{Z_t^{i,N,\vv}\neq0\}}\big(b(X_t^i,\tilde\mu_t^N)-b(X_t^{i,N,\vv},\tilde\mu_t^N)\big)\\
 &\quad+\frac{Z_t^{i,N,\vv}}{|Z_t^{i,N,\vv}|}\I_{\{Z_t^{i,N,\vv}\neq0\}}\big(b(X_t^{i,N,\vv},\tilde\mu_t^N)-b(X_t^{i,N,\vv},\hat\mu_t^{N,\vv})\big)\\
 &\le  \frac{1}{2}(\lambda_1+\lambda_2)|Z_t^{i,N,\vv}| \I_{\{| Z_t^{i,N,\vv} |\le\ell_0\}} -\frac{1}{2}\lambda_2| Z_t^{i,N,\vv}| +\frac{\lambda_3 }{N}\sum_{j=1}^N|Z_t^{j,N,\vv}|.
 \end{split}
\end{equation}
Using  the fact that  for $\mu\in\mathscr P_1(\R^d)$ and $x\in\R^d,$
\begin{align*}
\mathbb W_1\Big(\frac{N-1}{N}\mu+\frac{1}{N}\delta_x,\mu\Big)\le \frac{1}{N}\big(|x|+\mu(|\cdot|)\big),
\end{align*}
which can be attainable  analogously as \cite[Lemma 3.1]{Szpruch},  in addition to \eqref{EE1} and
\begin{align*}
\tilde\mu_t^{N}=\frac{N-1}{N}\tilde{\mu}_t^{N,i}+\frac{1}{N}\delta_{X_t^i}, \quad t\ge0, \, i\in\mathbb S_N,
\end{align*}
implies that
\begin{equation}\label{E26}
\begin{split}
\big|b(X_t^i,\tilde\mu_t^{N})-b(X_t^i,\tilde{\mu}_t^{N,i})\big|&\le \lambda_3\mathbb W_1(\tilde{\mu}_t^{N,i},\tilde\mu_t^{N})\\
&\le \frac{\lambda_3}{N}\bigg(|X_t^i|+\frac{1}{N-1}\sum_{j=1:j\neq i}^N|X_t^j|\bigg).
\end{split}
\end{equation}
Thus, plugging \eqref{E25} and \eqref{E26} back into   \eqref{E27}  yields that
\begin{equation*}\label{EE2}
\begin{split}
\d |Z_t^{i,N,\vv}|
&\le \frac{1}{2}\Big ( (\lambda_1+\lambda_2 )|Z_t^{i,N,\vv}| \I_{\{| Z_t^{i,N,\vv} |\le\ell_0\}} - \lambda_2| Z_t^{i,N,\vv}|  \Big)\,\d t\\
&\quad+ \bigg( J^i({\bf X}_t^N)+\frac{ \lambda_3 }{ N}\sum_{j=1}^N|Z_t^{j,N,\vv}|\bigg)\,\d t
+\d  \tilde M_t^{i,N,\vv},
\end{split}
\end{equation*}
where
\begin{align*}
J_i({\bf X}_t^N):=\frac{\lambda_3}{N}\bigg(|X_t^i|+\frac{1}{N-1}\sum_{j=1:j\neq i}^N|X_t^j|\bigg)\,\d t+ \big|b(X_t^i,\mu_t^i)-b(X_t^i,\tilde{\mu}_t^{N,i})\big|.
\end{align*}

Define the function
\begin{align*}
f(r) =1-\e^{-c_1r}+c_2r,\qquad  r\ge0,
\end{align*}
where $c_1,c_2>0$ were defined in \eqref{E0}.
Next,
 applying It\^o's formula to the function $f$ above,
we deduce  from $f'>0$ and $f''<0$  that
\begin{equation*}
\begin{split}
\d \big(\e^{\lambda_0^* t}f(| Z_t^{i,N,\vv}|)\big)
\le\e^{\lambda_0^* t}\bigg(&\lambda_0^* f(| Z_t^{i,N,\vv}|)+\psi(| Z_t^{i,N,\vv}|) h_\vv\big(\|{\bf Z}_t^{N,N,\vv}\|_1\big)^2\I_{\{Z_t^{i,N,\vv}\neq 0\}}+\Upsilon_i({\bf Z}_t^{N,N,\vv})\\
 &+f'(| Z_t^{i,N,\vv}|)\Big(J_i({\bf X}_t^N)+\frac{\lambda_3 }{ N}\sum_{j=1}^N|Z_t^{j,N,\vv}|\Big)\bigg)\,\d t +\d  \bar M_t^{i,N,\vv}
\end{split}
\end{equation*}
for some martingale $(\bar M_t^{i,N,\vv})_{t\ge0}$,
where $\lambda_0^*>0$ was defined in \eqref{E28},
\begin{align*}
\psi(r):&=\frac{1}{2}f'(r)\big((\lambda_1+\lambda_2)\I_{\{r\le\ell_0\}}-  \lambda_2 \big)r+2(\si_0^2+\sigma^2_1)  f''(r),\quad r\ge0
\end{align*}
and
\begin{align*}
\Upsilon_i({\bf Z}_t^{N,N,\vv}):=\frac{1}{2}f'(| Z_t^{i,N,\vv}|)\Big( (\lambda_1+\lambda_2 )|Z_t^{i,N,\vv}| \I_{\{| Z_t^{i,N,\vv} |\le\ell_0\}} - \lambda_2| Z_t^{i,N,\vv}|\Big)\big (1-h_\vv\Big(\|{\bf Z}_t^{N,N,\vv}\|_1\big)^2\Big).
\end{align*}

By virtue of
\begin{equation*}
 f'(r)= c_1\e^{-c_1r}+c_2, \quad \quad  f''(r)= -c_1^2\e^{-c_1r},\qquad r\ge0,
 \end{equation*}
and  the alternatives of $c_1$ and $c_2$ given in \eqref{E0}, for any $r\le\ell_0,$ we have
\begin{align*}
\psi(r)
\le  -c_1^2\e^{-c_1\ell_0}      (\si_0^2+\sigma_1^2) \le -\frac{c_1^2\e^{-c_1\ell_0}      (\sigma_0^2+\sigma^2_1) }{1-\e^{-c_1\ell_0}+c_2\ell_0}f(r)=-\frac{c_1c_2      (\sigma_0^2+\sigma^2_1) }{1-\e^{-c_1\ell_0}+c_2\ell_0}f(r).
\end{align*}
On the other hand, for the case $r\ge \ell_0,$ we infer  that
\begin{align*}
\psi(r)
=-\frac{1}{2}  \lambda_2 (c_1\e^{-c_1r}+c_2)r-2c_1^2(\si_0^2+\sigma_1^2) \e^{-c_1r}
&\le - \frac{\lambda_2 c_2r}{2(1-\e^{-c_1 r}+c_2r)} f(r)\\
&\le -\frac{c_2  \lambda_2 \ell_0}{2(1-\e^{-c_1\ell_0}+c_2\ell_0)}f(r),
\end{align*} where in the last inequality we used the fact that the function $r\mapsto\frac{r}{1-\e^{-c_1 r}+c_2r}$ is increasing on $(0,\infty)$.
Therefore, we arrive at
\begin{equation*}
\psi(r)\le -\bigg(\frac{c_1c_2     (\si_0^2+\sigma_1^2) }{1-\e^{-c_1\ell_0}+c_2\ell_0}\wedge \frac{c_2  \lambda_2 \ell_0}{2(1-\e^{-c_1\ell_0}+c_2\ell_0)}\bigg)f(r)=:-\lambda_0^{**} f(r),\quad r\ge0.
\end{equation*}
This, along with  $c_2\le f'(r)\le c_1+c_2$,  implies  that
\begin{equation*}
\begin{split}
 \e^{\lambda_0^*  t}\E f(| Z_t^{i,N,\vv}|)
\le \E f(| Z_0^{i,N,\vv}|)
+\int_0^t\e^{\lambda_0^*  s}\bigg[ & - C_1\bigg(\E f(| Z_s^{i,N,\vv}|)- \frac{1}{N}  \sum_{j=1}^N\E f(|Z_s^{j,N,\vv}|)\bigg)\\
&+\lambda_0^{**}\E\bigg(f(| Z_s^{i,N,\vv}|)\Big (1-h_\vv\big(\|{\bf Z}_s^{N,N,\vv}\|_1\big)^2\Big)\bigg)\\
&+(c_1+c_2)\E J_i({\bf X}_s^N)+\E\Upsilon_i({\bf Z}_s^{N,N,\vv})\bigg]\,\d s,
\end{split}
\end{equation*}
where $C_1:=\lambda_3(1+c_1/c_2)$. Next,
combining with
\begin{align*}
\E J_i({\bf X}_t^N)\le & \frac{\lambda_3}{N}\bigg(\E|X_t^i|+\frac{1}{N-1}\sum_{j=1:j\neq i}^N\E|X_t^j|\bigg)+ \varphi(N)\\
\le & \frac{C_2}{N}\big(1+\E|X_0^1|\big)+ \varphi(N)
\end{align*}
for some constant $C_2>0$, thanks to  \eqref{E23} and Lemma \ref{lem1},
 we deduce that
for some  $C_3>0,$
\begin{equation*}
\begin{split}
 \e^{\lambda_0^* t}\frac{1}{N}\sum_{i=1}^N\E f(| Z_t^{i,N,\vv}|)
&\le \frac{1}{N}\sum_{i=1}^N\E f(| Z_0^{i,N,\vv}|)
 + C_3\bigg( \frac{1}{N}\big(1+\E |X_0^1|\big) +\varphi(N) \bigg)\int_0^t\e^{\lambda_0^* s}\,\d s\\
&\quad+\lambda_0^{**}\int_0^t\e^{\lambda_0^* s}\E\bigg(\frac{1}{N}\sum_{i=1}^Nf(| Z_s^{i,N,\vv}|)\Big (1-h_\vv\big(\|{\bf Z}_s^{N,N,\vv}\|_1\big)^2\Big)\bigg)\,\d s\\
&\quad+\frac{1}{N}\sum_{i=1}^N\int_0^t\e^{\lambda_0^* s}\E\Upsilon_i({\bf Z}_s^{N,N,\vv})\,\d s.
\end{split}
\end{equation*}

By invoking  $c_2\le f'(r)\le c_1+c_2$ and $f(0)=0$, in addition to $h_\vv\in[0,1]$,  we find that for all $s>0$,
\begin{equation*}
\begin{split}
&\frac{1}{N}\sum_{i=1}^Nf(| Z_s^{i,N,\vv}|)\Big (1-h_\vv\big(\|{\bf Z}_s^{N,N,\vv}\|_1\big)^2\Big)+\frac{1}{N}\sum_{i=1}^N  \Upsilon_i({\bf Z}_s^{N,N,\vv})\\
&\le  (c_1+c_2)(2+\lambda_1
) \|{\bf Z}_s^{N,N,\vv}\|_1\big( 1-h_\vv\big( \|{\bf Z}_s^{N,N,\vv}\|_1\big) \big)\\
&\le 2(c_1+c_2)(2+\lambda_1
)\vv,
\end{split}
\end{equation*}
where in the last display we used the fact that
\begin{align*}
r(1-h_\vv(r))\le 2\vv,\qquad r\ge0
\end{align*}
by taking the definition of the function $h_\vv$ into consideration.  Thus, we derive that for some constant  $C_4>0$,
\begin{equation*}
\begin{split}
 \e^{\lambda_0^* t}\frac{1}{N}\sum_{i=1}^N\E f(| Z_t^{i,N,\vv}|)
&\le \frac{1}{N}\sum_{i=1}^N\E f(| Z_0^{i,N,\vv}|)
  + C_4\bigg(\frac{1}{N}\big(1+\E |X_0^1|\big) +\varphi(N)+\vv \bigg)\e^{\lambda_0^* t}.
\end{split}
\end{equation*}
Consequently, according to $c_2\le f'(r)\le c_1+c_2$ again and $f(0)=0$, there is a constant $C_5>0$ so that for all $t>0$, $$\frac{1}{N}\sum_{i=1}^N\E | Z_t^{i,N,\vv}| \le  \e^{-\lambda_0^*t}\frac{C_5}{N}\sum_{i=1}^N\E | Z_0^{i,N,\vv}|+C_5\bigg(\frac{1}{N}\big(1+\E |X_0^1|\big)+\varphi(N)+ \vv\bigg),$$ and so the desired assertion follows directly.
\end{proof}

Before we proceed, we make an  additional  comment.
\begin{remark}
We turn to the case that $d\ge 2$. Observe that the functions $\rho $ and $\varphi $ involved in $\phi_{d,\vv}$ are undetermined.
 By applying the It\^o-Tanaka formula to the radial process $|Z_t^{i,N,\vv}|$, it is easy to see that the   quadratic variation term
\begin{equation*}
\Upsilon_i(t):=\frac{2}{ |Z_t^{i,N,\vv}|^3}\I_{\{Z_t^{i,N,\vv}\neq{\bf0}\}}h_\vv(\rho({\bf Z}_t^{N,N,\vv})))^2\<|Z_t^{i,N,\vv}|^2I_d-Z_t^{i,N,\vv}\otimes Z_t^{i,N,\vv}, {\bf n}(\varphi({\bf Z}_t^{N,N,\vv}))\otimes {\bf n}(\varphi({\bf Z}_t^{N,N,\vv}))\>_{\rm HS}
\end{equation*}
arises naturally, where ${\bf Z}_t^{N,N,\vv} :={\bf X}_t^{N}-{\bf X}_t^{N,N,\vv}.$
Obviously,
in order to kill the term $\Upsilon_i(t)$ for any $i\in\mathbb S_N$ as needed in the proof above,   we cannot  choose $\varphi({\bf Z}_t^{N,N,\vv})$, which are dependent on all particles, by noting that the term $|Z_t^{i,N,\vv}|^2I_d-Z_t^{i,N,\vv}\otimes Z_t^{i,N,\vv}$ appears  in the inner product. Nevertheless, provided that  we take $\varphi({\bf Z}_t^{N,N,\vv})$, which is dependent merely on the $i$-th  component $Z_t^{i,N,\vv}$, the proof of Proposition \ref{pro1} is unavailable; see Remark \ref{remark:coup} for more explanations. The above further explains
why we focus merely on the $1$-dimensional SDE \eqref{E1} rather than the multi-dimensional setting.
\end{remark}

Based on the previous warm-up preparations, we start to complete the
\begin{proof}[Proof of Theorem $\ref{thm2}$]
Given $\mu,\nu\in \mathscr P_1(\R)$,   by
the
existence of optimal couplings,  there is $\pi^*\in\mathscr C(\mu,\nu)$ such that
\begin{equation}\label{E29}
\mathbb W_1(\mu,\nu)=  \int_{\R\times\R}|x-y|\pi^*(\d x,\d y).
\end{equation} Let
  $(X_t^{i,\mu})_{t\ge0}$ and $(X_t^{i,N,\nu})_{t\ge0}$ be respective solutions to \eqref{E6} and \eqref{E7}, where   $(X_0^{i,\mu},X_0^{i,N,\nu} )_{1\le i\le N}$
 are i.i.d. $\mathscr F_0^1$-measurable random variables  
 such that  $\mathscr L_{(X_0^i,X_0^{i,N,\nu})}=\pi^* $.  In particular, $\mathbb W_1(\mu,\nu)= \E|X_0^i-X_0^{i,N}|$, and
 the common distributions of $X_0^{i,\mu}$ and $X_0^{i,N,\nu}$ are just $\mu$ and $\nu$, respectively.

Via the triangle inequality, it is easy to see that for all $t>0$,
\begin{equation}\label{e:ppff1}\begin{split}
&\mathcal W_1(\mu_t,\nu_t)
\le \E^0\mathbb W_1(\mu_t,\nu_t)\\
&\le \E^0 \left(\E^1\mathbb W_1\bigg(\mu_t,\frac{1}{N}\sum_{j=1}^N\delta_{X_t^{j,\mu}}\bigg)\right)+\E^0\left(\E^1\mathbb W_1\bigg(\frac{1}{N}\sum_{j=1}^N\delta_{X_t^{j,\mu}},\frac{1}{N}\sum_{j=1}^N\delta_{X_t^{j,N,\nu}}\bigg)\right)\\
&\quad+\E^0\left(\E^1\mathbb W_1\bigg(\frac{1}{N}\sum_{j=1}^N\delta_{X_t^{j,N,\nu}},\frac{1}{N}\sum_{j=1}^N\delta_{X_t^{j,\nu}}\bigg)\right) +\E^0\left(\E^1\mathbb W_1\bigg(\nu_t,\frac{1}{N}\sum_{j=1}^N\delta_{X_t^{j,\nu}} \bigg)\right)\\
&= \E \mathbb W_1\bigg(\mu_t,\frac{1}{N}\sum_{j=1}^N\delta_{X_t^{j,\mu}}\bigg)+\E\mathbb W_1\bigg(\frac{1}{N}\sum_{j=1}^N\delta_{X_t^{j,\mu}},\frac{1}{N}\sum_{j=1}^N\delta_{X_t^{j,N,\nu}}\bigg)\\
&\quad+\E\mathbb W_1\bigg(\frac{1}{N}\sum_{j=1}^N\delta_{X_t^{j,N,\nu}},\frac{1}{N}\sum_{j=1}^N\delta_{X_t^{j,\nu}}\bigg) +\E\mathbb W_1\bigg(\nu_t,\frac{1}{N}\sum_{j=1}^N\delta_{X_t^{j,\nu}} \bigg)\\
&=:\Gamma_1(t,N)+\Gamma_2(t,N)+\Gamma_3(t,N)+\Gamma_4(t,N).
\end{split}\end{equation}

In the subsequent analysis, we estimate the terms $\Gamma_i(t,N), i=1,2,3,4$, separately.
Obviously, the Assumption (${\bf H}_{\bar\sigma}')$ implies the Assumption (${\bf A}_{\bar\sigma})$.
Note that for all $i\in \mathbb S_N$, $(X_t^{i,\mu})_{t\ge0}$  (resp. $(X_t^{i,N,\nu})_{t\ge0}$) shares the same initial value so that
$\E|X_0^{i,\mu}|<\infty$ (resp. $\E|X_0^{i,\nu}|<\8$).
Therefore,  an application of  Proposition \ref{pro0} yields that
\begin{align*}
\lim_{N\to\8}\big(\Gamma_1(t,N)+\Gamma_4(t,N)\big)=0.
\end{align*}
Next, note that
\begin{align*}
\Gamma_3(t,N)\le\frac{1}{N}\sum_{j=1}^N\E|X_t^{j,\nu}-X_t^{j,N,\nu}|=\E |X_t^{1,\nu}-X_t^{1,N,\nu}|,
\end{align*}
where the identity is due to the fact that $(X_t^{i,\nu},X_t^{i,N,\nu})$ and $(X_t^{j,\nu},X_t^{j,N,\nu})$ are identically distributed
thanks to   $(X_0^{i,\mu},X_0^{i,N,\nu} )_{1\le i\le N}$
  are i.i.d. $\mathscr F_0^1$-measurable random variables.  
  Whereafter, applying Proposition \ref{pro0} once more enables us to derive that
\begin{align*}
\lim_{N\to\8}\Gamma_3(t,N)=0.
\end{align*}

Consider the system \eqref{EE0} associated with the processes $(X_t^{i,\mu})_{t\ge0}$ and $(X_t^{i,N,\nu})_{t\ge0}$, which are respective solutions to \eqref{E6} and \eqref{E7}.
Denote by $({\bf X}^{N},{\bf X}^{N,N,\vv})$ the solution to the system \eqref{EE0}. 
Evidently, the Assumptions  $({\bf H}_{b,1})$ and $({\bf H}_{\bar\sigma}')$
imply $({\bf A}_{b })$ and $({\bf A}_{\bar\sigma})$. So, according to Proposition \ref{pro1}, $({\bf X}^N, {\bf X}^{N,N,\vv})$ has a weakly convergent subsequence such that the corresponding weak limit process is the coupling process of ${\bf X}^N$ and ${\bf X}^{N,N}$. In the following analysis, for the sake of notation simplicity,  we shall still write $({\bf X}^N,{\bf X}^{N,N})$ as the associated weak limit process.  Furthermore, it is ready to see that  there exists a constant $\lambda_3^*<\lambda_2/2$ such that \eqref{E28} is true for any $\lambda_3\in[0,\lambda_3^*].$ Thus, employing Proposition \ref{pro3}
and  $X_0^{i,N,\vv}=X_0^{i,N}$,
 we derive that there exists a  constant $C^\star>0$ such that
\begin{align*}
 \Gamma_2(t,N)\le &C^\star \Big(\e^{-\lambda_0^*t} \E|X_0^{i,\mu}-X_0^{i,N,\nu}|+\frac{1}{ N}\big(1+\E |X_0^1|\big)+\varphi(N)\Big)\\
 =& C^\star \Big(\e^{-\lambda_0^*t} \mathbb W_1(\mu,\nu)
 +\frac{1}{ N}\big(1+\E |X_0^1|\big)+\varphi(N)\Big).
\end{align*}
This, together with the prerequisite $\lim_{N\to\8}\varphi(N)=0$, leads to
\begin{align*}
\limsup_{N\to\8} \Gamma_2(t,N)\le C^\star\e^{-\lambda_0^*t}
\mathbb W_1(\mu,\nu).
\end{align*}

At last, by putting together  the estimates concerning  $\Gamma_i(t,N)$, $i=1,\cdots,4,$  we accomplish the proof of Theorem \ref{thm2}.
\end{proof}

We now can present  the proof of Theorem \ref{thm1} on the basis of Theorem \ref{thm2}.
\begin{proof}[Proof of Theorem $\ref{thm1}$]
As we elaborated in the second paragraph of this section, in order to investigate   ergodicity  of the measure-valued process $(\mu_t)_{t>0}$ associated with  \eqref{E1}, it is sufficient to consider the McKean-Vlasov SDE with common noise \eqref{EE6}. Based on Theorem \ref{thm2}, it remains to examine
the Assumptions imposed in Theorem \ref{thm2} with $\si_1=\ss{\alpha\kk_{\sigma,1}}$ and $\bar \sigma(x)=\bar \sigma_\alpha(x)$, separately. Concerning the drift $b$, the same assumptions are set in Theorems \ref{thm1} and \ref{thm2}. So, the validation on the drift $b$ is trivial.

Define the set
\begin{equation*}
\Lambda_\sigma=\Big\{\alpha>0: \inf_{x\in\R}\bar\sigma_\alpha(x)>0\Big\},
\end{equation*}
where  $\bar\sigma_\aa(x)=(\si(x)^2-\aa\kk_{\sigma,1})^{{1}/{2}}$ (see \eqref{EE7} for details). Below, we fix $\alpha\in\Lambda_\sigma$.
By virtue of  (${\bf H}_\sigma$), we deduce that  for   $x,y\in\R,$
\begin{align*}
|\bar\si_\aa(x)-\bar\si_\aa(y)|\le 2\ss{\kk_{\sigma,2}},
\end{align*}
and that for $x,y\in \R$,
\begin{align*}
|\bar\si_\aa(x)-\bar\si_\aa(y)|=\frac{|\si(x)^2-\si(y)^2|}{\bar\si_\aa(x)+\bar\si_\aa(y)}&\le \frac{(|\si(x)|+|\si(y)|)|\si(x)-\si(y)|}{\bar\si_\aa(x)+\bar\si_\aa(y)}\\
&\le \frac{L_\sigma\ss{\kk_{\sigma,2}}}{\inf_{x\in\R}\bar\si_\alpha(x)}|x-y|.
\end{align*}
Therefore, we arrive at
\begin{align*}
|\bar\si_\aa(x)-\bar\si_\aa(y)|\le \bigg(\big(2\ss{\kk_{\sigma,2}}\big)\vee\frac{L_\sigma\ss{\kk_{\sigma,2}}}{\inf_{x\in\R}\bar\si_\alpha(x)}\bigg)\big(1\wedge |x-y|\big),\quad x,y\in\R.
\end{align*}
Whence,  the Assumption (${\bf H}_{\bar\sigma_\alpha}'$) holds true with
$$L_{\bar\si_\alpha}'=\big(2\ss{\kk_{\sigma,2}}\big)\vee\frac{L_\sigma\ss{\kk_{\sigma,2}}}{\inf_{x\in\R}\bar\si_\alpha(x)}.$$
Furthermore, with $\si_1=\ss{\alpha\kk_{\sigma,1}}$ and $\bar \sigma(x)=\bar \sigma_\alpha(x)$ at hand, there exists a positive constant $\lambda_3^*<\lambda_2/2$ such that $\lambda_0^*>0$
for all $\lambda_3\in(0,\lambda_3^*]$, where $\lambda_0^*$ was introduced in \eqref{E28}.

In a word, all of the sufficiency conditions in Theorem \ref{thm2} are fulfilled  and therefore the proof of Theorem \ref{thm1} is complete.
\end{proof}

Before the end of this section, we make some further comments on the comparison concerned with our main result and  that in \cite[Section 4]{Maillet} for $d=1$, and our approach for the high dimensional setting (i.e., for $d\ge2$).

\begin{remark}\label{R:3.6} We compare
Theorem \ref{thm1} with the counterpart of \cite[Section 4]{Maillet} based on the following four aspects:
\begin{itemize}
\item {\it Framework}: In \cite[Section 4]{Maillet}, the drift $b(x,\mu)=-V'(x)+\int_{\R}W'(x-y)\mu(\d y)$, where both $V'$ and $W'$ are of {\it linear growth}. Whereas, in our setting, the drift $b$ is much more general and is allowed to be of polynomial growth with respect to the spatial variables. Moreover, in \cite[Section 4]{Maillet}, the   idiosyncratic noise is additive. However, in the present work, the  idiosyncratic noise is multiplicative.

\item {\it Contribution of noises}: As shown in Proposition \ref{pro3} and Remark \ref{Remark--ss},
not only the common noise but also the idiosyncratic noise make contributions to the exponential ergodicity of the measure-valued process $(\mu_t)_{t>0}$. Nevertheless, in \cite[Section 4]{Maillet}, the common noise makes the sole contribution to the ergodicity of $(\mu_t)_{t>0}$.

\item {\it Construction of the asymptotic coupling by reflection}: In  general, we can decompose the noise part to construct (asymptotic) coupling by reflection when the underlying  SDEs (including McKean-Vlasov SDEs) are partially dissipative as indicated in \cite[Section 4]{Maillet}. However, regarding McKean-Vlasov SDEs with common noise, if we adopt the previous procedure, then the  common noise will become not explicit and moreover
     change  drastically so the measure-valued process $(\mu_t)_{t>0}$ will satisfy a different nonlinear stochastic Fokker-Planck equation.
    Moreover, in order to carry out the proof of \cite[Theorem 2]{Maillet}, the identity \cite[(26)]{Maillet} is vital.
    Unfortunately, there is a gap to derive \cite[(26)]{Maillet} by invoking  the following  SDE (see \cite[Proposition 6]{Maillet} for more details):
    \begin{align}\label{EE}
    d |E_t^{i,N,\delta}|=-e_t^{i,N,\delta}\big(V'(X_t^{i,\delta})-V'(X_t^{i,N,\delta})\big)\d t+A_t^{i,N,\delta}\d t+2\si_0\pi_\delta(E_t^{N,\delta})(e_t^{i,N,\delta})^T\d B_t^0 ,
    \end{align}
where   $\pi_\delta(E_{t}^{N,\delta})^2\I_{\{E_t^{i,N,\delta}\neq 0\}}\neq  \pi_\delta(E_{t}^{N,\delta})^2$. Most importantly, we would like to emphasize that, unlike \cite[Lemma 7]{DEGZ},   the SDE \eqref{EE} cannot be derived via an approximate strategy as shown in
    \cite[Appendix A.5]{Maillet}, where in particular the identity in \cite[p.\ 28]{Maillet} is not true since the variables involved in functions $\pi_\delta$
   and $\psi_a$  are not consistent.
 Based on previous viewpoints, we build a totally novel asymptotic coupling by reflection as demonstrated in \eqref{EE0}.

 \item {\it Moment on initial distributions}:     To investigate ergodicity of $(\mu_t)_{t>0}$ under the Wasserstein distance $\mathcal W_1,$ it is quite reasonable to require $\mathscr L_{\mu_0}\in L_1(\mathscr P(\R))$, which is imposed in Theorem \ref{thm1}. However, $\mathscr L_{\mu_0}\in L_4(\mathscr P(\R))$ was set in \cite[Corollary 3]{Maillet} which is  concerned with exponential convergence under the $\mathcal W_1$-Wasserstein distance.
\end{itemize}
\end{remark}

\begin{remark}\label{remark} The proof of the exponential ergodicity for the measured process $(\mu_t)_{t>0}$ relies on  the inequality \eqref{e:ppff1}, where the terms $\Gamma_1(t,N)$, $\Gamma_3(t,N)$ and $\Gamma_4(t,N)$ can be handled similarly for $d\ge2$ due to the fact that Proposition \ref{pro1} holds for all $d\ge1$. Therefore, the main task is to estimate
$\Gamma_2(t,N)$. For this, once more,  we need to make use of the asymptotic coupling by reflection constructed in Subsection \ref{section2.2}. Different from the one-dimensional case, for the setting $d\ge2$,
  we can take  $\varphi({\bf x})= \bar{\bf x}:=\frac{1}{N}\sum_{j=1}^Nx_j$ and $\rho({\bf x})= |\bar{\bf x} |$. Note that the averaged process $\bar Z_t^{N,\vv}:=\frac{1}{N}\sum_{j=1}^N(X_t^j-X_t^{j,N,\varepsilon})$ solves the following SDE:
\begin{align*}
\d \bar Z_t^{N,\vv} &=\frac{1}{N}\sum_{i=1}^N\big(b(X_t^i,\mu_t^i)-b(X_t^{i,N,\vv},\hat \mu_t^{N,\vv})\big)\,\d t\\
&\quad+2h_\vv(|\bar Z_t^{N,\vv}|){\bf n}(\bar Z_t^{N,\vv})\otimes {\bf n}(\bar Z_t^{N,\vv})\bigg(\frac{1}{N}\sum_{i=1}^N\si_1\,\d B_t^{1,i} +\si_0 \,\d W_t\bigg).
\end{align*}
Whence, to derive the long-term estimate on the quantity  $|\bar Z_t^{N,\vv}|$, a special structure concerning the drift $b$ (e.g., $b(x,\mu)=-x+b_0(x,\mu)$ for some $b_0:\R^d\times\mathscr P(\R^d)\to\R^d$)   need to be required, which undoubtedly restrict applications of the theory derived. Furthermore,
 to achieve our aim,  it is also necessary  to quantitatively estimate the uniform-in-time distance between   each component process $X_t^i $ (resp. $X_t^{i,N,\vv}$) and the averaged process $\frac{1}{N}\sum_{j=1}^N X_t^j $ (resp. $\frac{1}{N}\sum_{j=1}^N X_t^{j,N,\vv} $); see \cite[Proposition 8]{Maillet} for related details. Unfortunately, such an estimate necessitates to require $\si_1=0.$
  This further reduces   practical applications of  the main result. In particular, when $b(x,\mu)=-V'(x)-2\alpha\int_{\R^d}(x-y)\mu(\d y) $ and $\si_1=0,$
 where $V':\R^d\to\R^d$ is globally Lipschitz and $\alpha>0,$ \cite[Theorem 3]{Maillet} derived Theorem \ref{thm1} for the multi-dimensional setup.
 One can refer to \cite[Section 5]{Maillet} for related discussions.
\end{remark}

\section{Appendix}
This Appendix section is devoted to providing a sufficiency condition to guarantee that the Assumption $({\bf H}_{b,2})$ is valid.

\begin{lemma}
Let   $(X^i_t)_{1\le i\le N}$ be conditionally independent and identically distributed    under the filtration $\mathscr F_t^W$ and
 $b(x,\mu)=\displaystyle\int_{\R}b_0(x-y)\,\mu(\d y)$ for some Lipschitz continuous function $b_0:\R\to\R$. Then, there exists a constant $C_0>0$ such that for all $i\in\mathbb S_N,$
\begin{align}\label{E10}
 \E|b(X^i_t,\mu_t^i)-b(X^i_t,\tilde \mu^{N,i}_t)|^2\le \frac{C_0}{N}\big(1+\E|X_t^i|^2\big),
\end{align} where $\tilde \mu^{N,i}_t:=\frac{1}{N-1}\sum_{j=1:j\neq i}^N\delta_{X^j_t}$.
In particular, $({\bf H}_{b,2})$ holds true with
  $$\varphi(N):=\frac{\ss{C_0}}{\ss N}\Big(1+\sup_{t\ge0}(\E|X_t^i|^2)^{{1}/{2}}\Big) $$
in case of $\sup_{t\ge0}(\E|X_t^i|^2)^{{1}/{2}}<\8.$
\end{lemma}

\begin{proof}
Obviously, the Assumption $({\bf H}_{b,2})$ is available provided that \eqref{E10} is attainable plus the validity of $\sup_{t\ge0}(\E|X_t^i|^2)^{{1}/{2}}<\8.$

Below, let $(X^i_t)_{1\le i\le N}$ be conditionally independent and identically distributed    under the filtration $\mathscr F_t^W$  and $\mu_t^i=\mathscr L_{X_t^i|\mathscr F_t^W}.$   Since
\begin{equation*}
b(X_t^i,\mu_t^i)=\frac{1}{N-1 }\sum_{j=1:j\neq i}^N\E\big( b_0( X_t^{i}- X_t^{j}) \big| X_t^i,\mathscr F_t^W\big),
\end{equation*}
we thus obtain that
\begin{align*}
\E|b(X^i_t,\mu_t^i)-b(X^i_t,\tilde \mu^{N,i})|^2=\frac{1}{(N-1)^2}\bigg(\sum_{j=1:j\neq i}^N\E|\Psi^{ij}_t|^2+\sum_{j,k=1:j,k\neq i, j\neq k}^N\E\big( \Psi^{ij}_t \Psi^{ik}_t\big) \bigg),
\end{align*}
where
\begin{align*}
\Psi^{ij}_t:=\E\big( b_0( X_t^{i}- X_t^{j}) \big| X_t^i,\mathscr F_t^0\big)-b_0( X_t^{i}- X_t^{j}).
\end{align*}
Notice that for any  $j,k\neq i$ and $j\neq k,$
\begin{align*}
 \E\big(\Psi^{ij}_t \Psi^{ik}_t\big)=\E\big(\E\big(\Psi^{ij}_t \Psi^{ik}_t\big| X_t^i,\mathscr F_t^0\big)\big)=\E\big(\E\big( \E(\Psi^{ij}_t| X_t^i,\mathscr F_t^0) \E(\Psi^{ik}_t| X_t^i,\mathscr F_t^0) \big)\big)=0
\end{align*}
by taking the conditional independency under $\mathscr F_t^W$ of the sequence $(X_t^i)_{0\le t\le N}$ into consideration.
Subsequently,
we derive   that
\begin{align*}
\E|b(X^i_t,\mu_t^i)-b(X^i_t,\tilde \mu^{N,i})|^2
&\le \frac{2}{(N-1)^2}\sum_{j=1:j\neq i}^N\E  \big|b_0( X_t^{i}- X_t^{j})\big|^2 \\
&\le  \frac{C_0}{(N-1)^2} \sum_{j=1 }^N\big(    \E| X_t^{j}|^2 +|b_0( 0 )|^2\big),
\end{align*}
where in the second inequality
 we utilized the Lipschitz property of $b_0$ and
the fact that $X_t^i$ and $X_t^j$ are identically distributed  given $\mathscr F_t^W$. Finally, \eqref{E10} follows directly by using again that $X_t^i$ and $X_t^j$ share the same law.
\end{proof}

\noindent {\bf Acknowledgements.}\,\,  The research of Jianhai Bao is supported by the National Key R\&D Program of China (2022YFA1006004) and the National Natural Science Foundation of China (No. 12071340).
The research of Jian Wang is supported by the National Key R\&D Program of China (2022YFA1006003) and the National Natural Science Foundations of China (Nos. 12071076 and 12225104).

\end{document}